\documentclass[a4paper,twoside,english,12pt]{amsart}

\usepackage{amsthm, amsmath, amssymb}
\usepackage[margin=4cm]{geometry} 
\usepackage{xy} \xyoption{all}

\theoremstyle{plain}
\newtheorem{theorem}{Theorem}[subsection]
\newtheorem{corollary}[theorem]{Corollary}
\newtheorem{proposition}[theorem]{Proposition}
\newtheorem{lemma}[theorem]{Lemma}
\newtheorem{sublemma}[theorem]{Sublemma}

\theoremstyle{definition}
\newtheorem{definition}[theorem]{Definition}
\newtheorem{example}[theorem]{Example}

\newtheorem{remark}[theorem]{Remark}

\numberwithin{equation}{subsection}

\makeatletter

\@namedef{subjclassname@2010}{%
  \textup{2010} Mathematics Subject Classification}

\renewcommand{\emptyset}{\varnothing}

\newcommand{\Union}{\bigcup\limits}

\newcommand{\C}{\mathbb{C}}
\newcommand{\N}{\mathbb{N}}

\newcommand{\Z}{\mathbb{Z}}

\newcommand{\op}{\mathrm{op}}

\DeclareMathOperator{\id}{id}

\newcommand{\BDC}{\mathbf{D}^{\mathrm{b}}}

\newcommand{\DSum}{\bigoplus}
\newcommand{\dsum}[1][]{\mathbin{\oplus_{#1}}}

\newcommand{\ilim}[1][]{\mathop{\varinjlim}\limits_{#1}}

\renewcommand{\to}[1][]{\xrightarrow{#1}}

\def\isoto{\@ifnextchar [{\relisoto}{\absisoto}}
\newcommand{\relisoto}[1][]{\xrightarrow[\sim]{#1}}
\newcommand{\absisoto}[1][]{\xrightarrow{\sim}}

\newcommand{\Endo}[1][]{\mathrm{End}_{\raise1.5ex\hbox to.1em{}#1}}

\newcommand{\Hom}[1][]{\mathrm{Hom}_{\raise1.5ex\hbox to.1em{}#1}}

\newcommand{\RHom}[1][]{\mathrm{RHom}_{\raise1.5ex\hbox to.1em{}#1}}

\newcommand{\Ext}[2][]{\mathrm{Ext}_{\raise1.5ex\hbox to.1em{}#1}^{#2}}

\newcommand{\Tens}[1][]{\mathbin{\otimes_{\raise1.5ex\hbox to-.1em{}#1}}}

\newcommand{\LTens}[1][]{\mathbin{\otimes_{\raise1.5ex\hbox to-.1em{}#1}^{L}}}

\newcommand{\Tor}[2][]{\mathrm{Tor}^{\raise1.5ex\hbox to.1em{}#1}_{#2}}

\newcommand{\sheaffont}[1]{\mathcal{#1}}
\def\sha{\sheaffont{A}}
\def\shb{\sheaffont{B}}

\def\she{\sheaffont{E}}
\def\shf{\sheaffont{F}}
\def\shg{\sheaffont{G}}

\def\shl{\sheaffont{L}}
\def\shm{\sheaffont{M}}
\def\shn{\sheaffont{N}}
\def\sho{\sheaffont{O}}

\def\shr{\sheaffont{R}}

\def\shw{\sheaffont{W}}

\newcommand{\sect}{\varGamma}

\renewcommand{\hom}[1][]{{\sheaffont{H}om}_{\raise1.5ex\hbox to.1em{}#1}}
\newcommand{\aut}[1][]{{\sheaffont{A}ut}_{\raise1.5ex\hbox to.1em{}#1}}
\newcommand{\inn}[1][]{{\sheaffont{I}nn}_{\raise1.5ex\hbox to.1em{}#1}}

\newcommand{\rhom}[1][]{{R\sheaffont{H}om}_{\raise1.5ex\hbox to.1em{}#1}}

\newcommand{\ext}[2][]{{\sheaffont{E}xt}_{\raise1.5ex\hbox to.1em{}#1}^{#2}}

\newcommand{\thom}[1][]{{T\sheaffont{H}om}_{\raise1.5ex\hbox to.1em{}#1}}

\newcommand{\tens}[1][]{\mathbin{\otimes_{\raise1.5ex\hbox to-.1em{}#1}}}

\newcommand{\ltens}[1][]{\mathbin{\otimes_{\raise1.5ex\hbox to-.1em{}#1}^{L}}}

\newcommand{\tor}[2][]{{\sheaffont{T}or}^{\raise1.5ex\hbox to.1em{}#1}_{#2}}

\DeclareMathOperator{\supp}{supp}

\newcommand{\oim}[1]{{#1}_*}
\newcommand{\eim}[1]{{#1}_!}

\newcommand{\opb}[1]{#1^{-1}}

\newcommand{\GHom}[1][]{\mathrm{GHom}_{\raise1.5ex\hbox to.1em{}#1}}

\newcommand{\GExt}[2][]{\mathrm{GExt}_{\raise1.5ex\hbox to.1em{}#1}^{#2}}

\newcommand{\FHom}[1][]{\mathrm{FHom}_{\raise1.5ex\hbox to.1em{}#1}}

\newcommand{\ghom}[1][]{{\sheaffont{GH}om}_{\raise1.5ex\hbox to.1em{}#1}}

\newcommand{\gext}[2][]{{\sheaffont{GE}xt}_{\raise1.5ex\hbox to.1em{}#1}^{#2}}

\newcommand{\fhom}[1][]{{\sheaffont{FH}om}_{\raise1.5ex\hbox to.1em{}#1}}

\newcommand{\g}{\sheaffont{G}}

\newcommand{\tenstop}[1][]{\mathbin{\hat{\otimes}_{\raise1.5ex\hbox to-.1em{}#1}}}

\newcommand{\homtop}[1][]{\sheaffont{L}_{\raise1.5ex\hbox to.1em{}#1}}

\newcommand{\Homtop}[1][]{\mathrm{L}_{\raise1.5ex\hbox to.1em{}#1}}

\def\absdoim#1{\underline{#1}_*}
\def\reldoim[#1]#2{\underline{#2}_{|{#1}*}}
\def\doim{\@ifnextchar [{\reldoim}{\absdoim}}

\def\absdeim#1{\underline{#1}_*}
\def\reldeim[#1]#2{\underline{#2}_{|{#1}*}}
\def\deim{\@ifnextchar [{\reldeim}{\absdeim}}

\def\absdopb#1{\underline{#1}^{-1}}
\def\reldopb[#1]#2{\underline{#2}_{|{#1}}^{-1}}
\def\dopb{\@ifnextchar [{\reldopb}{\absdopb}}

\def\absboim#1{\underline{\underline{#1}}_*}
\def\relboim[#1]#2{\underline{\underline{#2}}_{|{#1}*}}
\def\boim{\@ifnextchar [{\relboim}{\absboim}}

\def\absbeim#1{\underline{\underline{#1}}_*}
\def\relbeim[#1]#2{\underline{\underline{#2}}_{|{#1}*}}
\def\beim{\@ifnextchar [{\relbeim}{\absbeim}}

\def\absbopb#1{\underline{\underline{#1}}^*}
\def\relbopb[#1]#2{\underline{\underline{#2}}_{|{#1}}^*}
\def\bopb{\@ifnextchar [{\relbopb}{\absbopb}}

\newcommand{\coh}{\mathrm{coh}}

\newcommand{\reghol}{\mathrm{r-hol}}

\DeclareMathOperator{\Supp}{Supp}

\newcommand{\rhof}{\rho\text{-\rm f}}

\renewcommand{\reghol}{\mathrm{rh}}

\newcommand{\filt}[2][]{F_{#1}#2}

\newcommand{\Ga}{\C}
\newcommand{\GRGa}{R}
\newcommand{\Gm}{\C^\times}

\newcommand{\infGa}{v_{\mathrm a}}
\newcommand{\infGm}{v_{\mathrm m}}

\newcommand{\multiindex}{J}

\newcommand{\hfield}{\mathbf{k}}
\newcommand{\hfieldo}{\hfield(0)}

\newcommand{\setdef}{;\ }

\newcommand{\h}{\hbar}
\newcommand{\hh}{\mathsf{t}}

\newcommand{\p}{p}
\newcommand{\q}{q}
\newcommand{\pq}{p}
\newcommand{\dq}{q}
\renewcommand{\g}{g}

\newcommand{\reg}{{\text{reg}}}

\newcommand{\ad}{\operatorname{ad}}
\newcommand{\Ad}{\operatorname{Ad}}

\newcommand{\shHom}[1][]{\sheaffont{H}om_{#1}}
\newcommand{\shIso}[1][]{\sheaffont{I}som_{#1}}
\newcommand{\shAut}[1][]{\sheaffont{A}ut_{#1}}
\newcommand{\shEnd}[1][]{\sheaffont{E}nd_{#1}}

\newcommand{\stack}[1]{\mathsf{#1}}

\newcommand{\stka}{\stack{A}}

\newcommand{\stkc}{\stack{C}}

\newcommand{\stke}{\stack{E}}

\newcommand{\stkp}{\stack{P}}
\newcommand{\stkq}{\stack{Q}}
\newcommand{\stkr}{\stack{R}}

\newcommand{\stkt}{\stack{T}}
\newcommand{\stkw}{\stack{W}}

\newcommand{\stkpe}{\stkp_Y}
\newcommand{\stkpw}{\stkp_X}
\newcommand{\stkpr}{\stkp_\rho}
\newcommand{\stkpa}{\stkp'_X}
\newcommand{\stkpao}{\stkp'_{X,0}}

\newcommand{\stkquant}{{\widetilde\stke}}
\newcommand{\shquant}{{\widetilde\she}}

\newcommand{\stkMod}[1][]{\stack{Mod}_{#1}}
\newcommand{\stkalg}[1][]{\stack{Alg}_{#1}}
\newcommand{\stkFun}[1][]{\stack{Fct}_{#1}}

\makeatother

\title{On quantization of complex symplectic manifolds}

\author[A. D'Agnolo]{Andrea D'Agnolo} 
\address{Dipartimento di Matematica Pura ed Applicata\\ 
Universit{\`a} di Padova\\ 
via Trieste 63, 35121 Padova, Italy} 
\email{dagnolo@math.unipd.it}

\author[M. Kashiwara]{Masaki Kashiwara} 
\address{Research Institute for Mathematical Sciences\\
Kyoto University\\
Kyoto, 606-8502, Japan}
\email{masaki@kurims.kyoto-u.ac.jp}

\thanks{The first named author (A.D'A.) expresses his gratitude to the Research Institute for Mathematical Sciences of
Kyoto University for hospitality during the preparation of this paper.}



\begin{document}

\maketitle

\begin{abstract}
Let $X$ be a complex symplectic manifold.
By showing that any Lagrangian subvariety has a unique lift to a
contactification,
we associate to $X$ a triangulated category of regular holonomic
microdifferential modules. If $X$ is compact, this is a Calabi-Yau category
of complex dimension $\dim X+1$.
We further show that regular holonomic microdifferential modules can
be realized as modules over a quantization algebroid canonically
associated to $X$.
\end{abstract}

\setcounter{tocdepth}{1}
\tableofcontents

\section*{Introduction}

Let $X$ be a complex symplectic manifold. As shown in~\cite{PS04} (see also~\cite{Kon01}), $X$ is endowed with a canonical
deformation quantization algebroid $\stkw_X$. Recall that an algebroid is to an
algebra as a gerbe is to a group.  The local model of $\stkw_X$ is an
algebra similar to the one of microdifferential operators, with a central
deformation parameter $\h$. The center of $\stkw_X$ is a subfield
$\hfield$ of formal Laurent series $\C[\h^{-1},\h]]$.

Deformation quantization modules have now been studied quite
extensively (see~\cite{DS07,KS08,KS10} and also \cite{NT04,Tsy09} for related results), and they turned out to be
useful in other contexts as well (see e.g.~\cite{KR08}).  Of particular
interest are modules supported by Lagrangian
subvarieties. It is conjectured in~\cite{KS08} that, if $X$ is compact, the triangulated category of regular
holonomic deformation-quantization modules is
Calabi-Yau of dimension $\dim X$ over $\hfield$. 

\medskip

There are some cases
(representation theory, homological mirror symmetry, quantization in the sense of~\cite{GW08}) where one would
like to deal with categories whose center is $\C$ instead of $\hfield$.
In the first part of this paper, we associate to $X$
a $\C$-linear triangulated category of regular holonomic
microdifferential modules. 
If $X$ is compact, this category is Calabi-Yau of dimension $\dim X+1$ over $\C$. 

Our construction goes as follows.  For a
possibly singular Lagrangian subvariety $\Lambda\subset X$, we prove that there is a
unique contactification $\rho\colon Y\to X$ of a neighborhood of
$\Lambda$ and a Lagrangian subvariety $\Gamma\subset Y$ such that
$\rho$ induces a homeomorphism between $\Gamma$ and $\Lambda$.  As
shown in~\cite{Kas96}, the contact manifold $Y$ is endowed with a
canonical microdifferential algebroid $\stke_Y$.  We define the triangulated category of regular
holonomic microdifferential modules along $\Lambda$ as the bounded derived category of regular
holonomic $\stke_Y$-modules along $\Gamma$. We then take the direct
limit over the inductive family of Lagrangian subvarieties
$\Lambda\subset X$.

\medskip

In the second part of this paper, we show that regular holonomic
microdifferential modules can be realized as modules over a quantization algebroid
$\stkquant_X$ canonically associated to $X$. 
More precisely, if $\Gamma\subset Y$ is a lift of $\Lambda\subset X$ 
as above, we prove that the category of coherent $\stke_Y$-modules supported on
$\Gamma$ is fully faithfully embedded in the category of coherent
$\stkquant_X$-modules supported on $\Lambda$.

Our construction of
$\stkquant_X$ is similar to the construction of $\stkw_X$
in~\cite{PS04}, which was in turn similar to the construction of
$\stke_Y$ in~\cite{Kas96}.  Here, we somewhat simplify matters by
presenting an abstract way of obtaining an algebroid from the
data of a gerbe endowed with an algebra valued functor.
Let us briefly recall the constructions of $\stke_Y$,
$\stkw_X$ and present the construction of $\stkquant_X$.

Denote by $P^*M$ the projective cotangent bundle to a complex manifold
$M$ and by $\she_M$ the ring of microdifferential operators on $P^*M$
as in~\cite{S-K-K}. Recall that, in a local system
of coordinates, $\she_M$ is endowed with the anti-involution given by
the formal adjoint of total symbols.

Let $Y$ be a complex contact manifold. By Darboux theorem, the local
model of $Y$ is an open subset of $P^*M$. By definition, a
microdifferential algebra $\she$ on an open subset $V\subset Y$ is a
$\C$-algebra locally isomorphic to $\she_M$. Assume that $\she$ is
endowed with an anti-involution $*$. Any two such pairs $(\she',*')$
and $(\she,*)$ are locally isomorphic. Such isomorphisms are not
unique, and in general it is not possible to patch the algebras $\she$
together in order to get a globally defined microdifferential algebra
on $Y$. However, the automorphisms of $(\she,*)$ are all inner and are
in bijection with a subgroup of invertible elements of $\she$. This is
enough to prove the existence of a microdifferential algebroid
$\stke_Y$, i.e.\ an algebroid locally represented by microdifferential algebras.

Denote by $T^*M$ the cotangent bundle to a complex manifold $M$, by
$(t;\tau)$ the symplectic coordinates on $T^*\C$, and consider the
projection
\[
P^*(M\times\C) \to[\rho] T^*M, \quad (x,t;\xi,\tau)\mapsto
(x,\xi/\tau)
\]
defined for $\tau\neq 0$.  This is a principal $\C$-bundle, with
action given by translation in the $t$ variable. Note that,
for $\lambda\in \C$, the outer
isomorphism $\Ad(e^{\lambda\partial_t})$ of $\oim\rho\she_{M\times\C}$
acts by translation
$t\mapsto t+\lambda$ at the level of total symbols.

Let $X$ be a complex symplectic manifold. By Darboux theorem, the
local model of $X$ is an open subset of $T^*M$. Let $\rho\colon V\to
U$ be a contactification of an open subset $U\subset X$. By
definition, this is a principal $\C$-bundle whose local model is the
projection $\{\tau\neq 0\}\to T^*M$ above. Consider a quadruple
$(\rho,\she,*,\h)$ of a contactification $\rho\colon V\to U$, a
microdifferential algebra $\she$ on $V$, an anti-involution $*$ and an
operator $\h\in\she$ locally corresponding to $\partial_t^{-1}$. One
could try to mimic the above construction of the microdifferential
algebroid $\stke_Y$ in order to get an algebroid from the algebras
$\oim\rho\she$. This fails because the automorphisms of
$(\rho,\she,*,\h)$ given by $\Ad(e^{\lambda\h^{-1}})$ for $\lambda\in \C$ 
are not inner.
There are two natural ways out.

The first possibility, utilized in~\cite{PS04}, is to replace the
algebra $\oim\rho\she$ by its subalgebra $\shw = C^0_\h\oim\rho\she$
of operators commuting with $\h$. Locally, this corresponds to the
operators of $\oim\rho\she_{M\times\C}$ whose total symbol does not
depend on $t$. Then the action of $\Ad(e^{\lambda\h^{-1}})$ is trivial
on $\shw$, and these algebras patch together to give the
deformation-quantization algebroid $\stkw_X$.

The second possibility, which we exploit here, is to make
$\Ad(e^{\lambda\h^{-1}})$ an inner automorphism. This is obtained by
replacing the algebra $\oim\rho\she$ by the algebra
\[
\shquant = \DSum_{\lambda\in\Ga} \bigl(C_\h^\infty\oim\rho \she\bigr)
\, e^{\lambda\h^{-1}},
\]
where $C_\h^\infty\oim\rho \she = \{a \in \oim\rho\she\setdef
\ad(\h)^N(a)=0,\ \exists N\geq 0 \}$ locally corresponds to operators
in $\oim\rho\she_{M\times\C}$ whose total symbol is polynomial in
$t$. By patching these algebras we get the quantization algebroid
$\stkquant_X$. The deformation parameter $\h$ is not central in
$\stkquant_X$. We show that the centralizer of $\h$ in $\stkquant_X$
is equivalent to the twist of $\stkw_X \tens[\C]
(\DSum\nolimits_{\lambda\in\C}\C e^{\lambda\h^{-1}})$ by the gerbe
parameterizing the primitives of the symplectic $2$-form.

\medskip

In an appendix at the end of the paper, we give an alternative
construction of the deformation-quantization algebroid
$\stkw_X$. Instead of using contactifications, we consider as objects
deformation-quantization algebras endowed with compatible
anti-involution and $\C$-linear derivation. We thus show that $\stkw_X$
itself is endowed with a canonical $\C$-linear derivation. One could
then easily prove along the lines of~\cite{Pol08} that $\stkw_X$ is the
unique $\hfield$-linear deformation-quantization algebroid which has trivial graded and is endowed
with compatible anti-involution and $\C$-linear derivation.  

Finally, we compare regular holonomic quantization modules with regular holonomic deformation-quantization modules.

\medskip

This paper is organized as follows.

In section~\ref{se:cat}, after recalling the definitions of gerbe and
of algebroid on a topological space, we explain how to obtain an
algebroid from the data of a gerbe endowed with an algebra valued
functor.

In section~\ref{se:geo}, we review some notions from contact and
symplectic geometry, discussing in particular the gerbe parameterizing
the primitives of the symplectic $2$-form.  We further show how a
Lagrangian subvariety lifts to a contactification.

In section~\ref{se:contact}, we first recall the construction of the
microdifferential algebroid of~\cite{Kas96} in terms of algebroid
data.  Then we show how to associate to a complex symplectic manifold
a triangulated category of regular holonomic microdifferential
modules.

In section~\ref{se:sympl}, we start by giving a construction of the
deformation-quantization algebroid of~\cite{PS04} in terms of
algebroid data.  Then, with the same algebroid data, we construct the
algebroid $\stkquant_X$.

In section~\ref{se:mod}, we prove coherency of quantization algebras
and show how to realize regular holonomic microdifferential modules as
modules over $\stkquant_X$.

In appendix~\ref{se:dq}, we give an alternative description of the
deformation quantization algebroid using deformation-quantization algebras endowed with compatible anti-involution and $\C$-linear derivation. 
We also compare regular holonomic deformation-quantization
modules with regular holonomic quantization modules.

\medskip

The results of this paper were announced in~\cite{DK10}, to which we refer.

\section{Gerbes and algebroid stacks}\label{se:cat}

We review here some notions from the theory of stacks, in the sense of
sheaves of categories, recalling in particular the definitions of
gerbe and of algebroid (refer to \cite{Gir71,KS06,Kon01,DP05}). We
then explain how to obtain an algebroid from the data of a gerbe
endowed with an algebra valued functor.

\subsection{Review on stacks}\label{sse:gerbes}

Let $X$ be a topological space.

A prestack $\stkc$ on $X$ is a lax analogue of a presheaf of
categories, in the sense that for a chain of open subsets $W\subset
V\subset U$ the restriction functor $\stkc(U)\to \stkc(W)$ coincides
with the composition $\stkc(U)\to \stkc(V)\to \stkc(W)$ only up to an
invertible transformation (satisfying a natural cocycle condition for
chains of four open subsets). The prestack $\stkc$ is called separated
if for any $U\subset X$ and any $\p,\p'\in\stkc(U)$ the presheaf
$U\supset V\mapsto \Hom[\stkc(V)](\p|_V,\p'|_V)$ is a sheaf. We denote
it by $\shHom[\stkc](\p,\p')$. A stack is a separated prestack
satisfying a natural descent condition (see e.g.~\cite[Chapter~19]{KS06}). If $\rho\colon Y\to X$ is a
continuous map, we denote by $\opb \rho\stkc$ the pull back on $Y$ of
a stack $\stkc$ on $X$.

A groupoid is a category whose morphisms are all invertible. A gerbe
on $X$ is a stack of groupoids which is locally non empty and locally
connected, i.e.\ any two objects are locally isomorphic. Let $\shg$ be
a sheaf of commutative groups. A $\shg$-gerbe is a gerbe $\stkp$
endowed with a group homomorphism $\shg \to \shAut(\id_\stkp)$. We
denote by $\stkp \times^\shg \stkp'$ the contracted product of two
$\shg$-gerbes. A $\shg$-gerbe $\stkp$ is called invertible if
$\shg|_U\to\shAut[\stkp](\p)$ is an isomorphism of groups for any
$U\subset X$ and any $\p\in\stkp(U)$.

Let $\shr$ be a commutative sheaf of rings. For an
$\shr$-algebra $\sha$ denote by $\stkMod(\sha)$ the stack of left
$\sha$-modules. An $\shr$-linear stack is a stack $\stka$ such that
for any $U\subset X$ and any $\p,\p'\in\stka(U)$ the sheaves
$\shHom[\stka](\p',\p)$ have an $\shr|_U$-module structure compatible
with composition and restriction.  The stack of left $\stka$-modules
$\stkMod(\stka) = \stkFun[\shr](\stka,\stkMod(\shr))$ has
$\shr$-linear functors as objects and transformations of functors as
morphisms.

Let $\shl$ be a commutative $\shr$-algebra and $\stka$ an
$\shr$-linear stack.  An action of $\shl$ on $\stka$ is the data of
$\shr|_U$-algebra morphisms $\shl|_U\to\shEnd[\stka](\p)$ for any
$U\subset X$ and any $\p\in\stka(U)$, compatible with
restriction. Then $\shl$ acts as a Lie algebra on
$\shHom[\stka](\p',\p)$ by $[l,f] = l_{\p} f - f l_{\p'}$, where $l_\p$
denotes the image of $l\in\shl(U)$ in $\shEnd[\stka](\p)$. This gives
a filtration of $\stka$ by the centralizer series
\begin{align*}
  C_\shl^0\shHom[\stka](\p',\p)
  &= \{f\setdef [l,f] = 0,\ \forall l\in\shl \}, \\
  C_\shl^{i}\shHom[\stka](\p',\p) &= \{f\setdef [l,f] \in C_\shl^{i-1},\
  \forall l\in\shl \} \quad \text{for any } i> 0.
\end{align*}
Denote by $C_\shl^0\stka$ and $C_\shl^\infty\stka$ the substacks of
$\stka$ with the same objects as $\stka$ and morphisms
$C_\shl^0\shHom[\stka]$ and $\Union\nolimits_i C_\shl^i\shHom[\stka]$,
respectively. Note that $C_\shl^0\stka$ is an $\shl$-linear stack and
$C_\shl^\infty\stka$ is an $\shr$-linear stack.

An $\shr$-algebroid $\stka$ is an $\shr$-linear stack which is locally
non empty and locally connected by isomorphisms. Thus, an algebroid is
to a sheaf of algebras as a gerbe is to a sheaf of groups. For
$\p\in\stka(U)$, set $\sha_\p = \shEnd[\stka](\p)$.  Then $\stka|_U$
is equivalent to the full substack of $\stkMod(\sha_p^\op)$ whose objects
are locally free modules of rank one. 
(Here $\sha_p^\op$ denotes the opposite ring of $\sha_p$.)
Moreover, there is an
equivalence $\stkMod(\stka|_U) \simeq \stkMod(\sha_\p)$.  One says
that $\stka$ is represented by an $\shr$-algebra $\sha$ if
$\sha\simeq\sha_p$ for some $\p\in\stka(X)$.  The $\shr$-algebroid
$\stka$ is called invertible if $\sha_\p\simeq\shr|_U$ for any
$U\subset X$ and any $\p\in\stka(U)$.

The pull-back and tensor product of algebroids are still
algebroids. The following lemma is obvious.

\begin{lemma}\label{lem:LiaAlg}
  Let $\stka$ be an $\shr$-algebroid endowed with an action of
  $\shl$. If $C_\shl^0\stka$ is locally connected by isomorphisms,
  then $C_\shl^0\stka$ and $C_\shl^\infty\stka$ are algebroids.
\end{lemma}

\subsection{Algebroid data}\label{sse:algdata}

Let $\shr\text{-}\stkalg$ be the stack on $X$ with $\shr$-algebras as
objects and $\shr$-algebra homomorphisms as morphisms.

\begin{definition}\label{def:lift}
  An $\shr$-algebroid data is a triple $(\stkp,\Phi,\ell)$ with
  $\stkp$ a gerbe, $\Phi\colon\stkp\to\shr\text{-}\stkalg$ a functor
  of prestacks and $\ell$ a collection of liftings of group
  homomorphisms
  \begin{equation}
    \label{eq:PhiAd}
    \vcenter{\xymatrix{
        & \Phi(\p)^\times \ar[d]^{\Ad} \\ \shEnd[\stkp](\p) \ar[r]^-\Phi
        \ar[ur]^{\ell_\p} & \shAut[{\shr\text{-}\stkalg}](\Phi(\p))
      }}
    \qquad \forall U\subset X,\ \forall \p\in\stkp(U),
  \end{equation}
  compatible with restrictions and such that for any
  $\g\in\shHom[\stkp](\p', \p)$ and any $\phi'\in\shEnd[\stkp](\p')$
  one has
  \begin{equation}
    \label{eq:PhiCond}
    \ell_\p(\g\phi' \g^{-1}) = \Phi(\g)(\ell_{\p'}(\phi')).
  \end{equation}
\end{definition}

Note that condition \eqref{eq:PhiCond} ensures compatibility with the
equality $\Phi(\g\phi' \g^{-1}) = \Phi(\g)\Phi(\phi')\Phi(\g^{-1})$.

\begin{remark}
  Denote by $\stack{Grp}$ the stack on $X$ with sheaves of groups as
  objects and group homomorphisms as morphisms. The $\shr$-algebroid
  data $(\stkp,\Phi,\ell)$ induce three natural functors $E,A,F\colon
  \stkp\to \stack{Grp}$ defined by $E(\p) = \shEnd[\stkp](p)$, $A(\p)
  = \shAut(\Phi(p))$ and $F(\p) = \Phi(p)^\times$ for $p\in\stkp$. In
  all three cases, a morphism $\p'\to\p$ is sent to its adjoint. Then
  the commutative diagram \eqref{eq:PhiAd} corresponds to a
  commutative diagram of transformations of functors
  \[
  \xymatrix@C=8ex{ & F \ar[d]^{\Ad} \\ E \ar[r]_-\Phi \ar[ur]^{\ell} & A.  }
  \]
\end{remark}

\begin{remark}
  There is a natural interpretation of $\shr$-algebroid data in terms
  of $2$-categories (refer to~\cite[\S9]{Str96}, where $2$-categories
  are called bicategories). Denote by $\shr\text{-}\mathbf{Alg}$ the
  $2$-prestack on $X$ obtained by enriching $\shr\text{-}\stkalg$ with
  set of $2$-arrows $f'\Rightarrow f$ given by
  \[
  \{b\in\sha \setdef bf'(a') = f(a')b,\ \forall a'\in\sha'\},
  \]
  for two $\shr$-algebra morphisms $f,f'\colon\sha'\to\sha$. In
  particular, $f\simeq f'$ if and only if $f' = \Ad(b)f$ for some
  $b\in\sha^\times$.  The $\shr$-algebroid data $(\stkp,\Phi,\ell)$ is
  equivalent to the data of the lax functor of $2$-prestacks
  \[ {\bf \Phi}\colon \stkp \to \shr\text{-}\mathbf{Alg},
  \]
  where $\stkp$ has trivial $2$-arrows and ${\bf \Phi}$ is obtained by
  enriching $\Phi$ at the level of $2$-arrows by ${\bf
    \Phi}(\id_{\g'\to \g}) = \ell_\p(\g'\g^{-1})$ for a
  morphism $\g'\to \g$ in $\stkp(\p)$.
\end{remark}

We will prove in the next proposition that the following description
associates an $\shr$-prestack $\stka_0$ to the data
$(\stkp,\Phi,\ell)$.

\begin{itemize}
\item[(i)] For an open subset $U\subset X$, objects of $\stka_0(U)$
  are the same as those of $\stkp(U)$.

\item[(ii)] For $\p,\p' \in \stka_0(U)$, the sheaf of morphisms is
  defined by
  \[
  \shHom[\stka_0]( \p', \p ) = \Phi(\p)
  \mathop{\times}\limits^{\shEnd[\stkp](\p)} \shHom[\stkp](\p',\p).
  \]
  This means that morphisms $\p'\to \p$ in $\stka_0$ are equivalence
  classes $[a,\g]$ of pairs $(a,\g)$ with $a\in\Phi(\p)$ and $\g\colon
  \p'\to\p$ in $\stkp$, for the relation
  \[
  (a,\,\phi \g) \sim (a \ell_\p(\phi),\,\g), \qquad \forall
  \phi\in\shEnd[\stkp](\p).
  \]

\item[(iii)] Composition of $[a,\g]\colon \p' \to \p$ and
  $[a',\g']\colon \p'' \to \p'$ is given by
  \[ [a,\g] \circ [a',\g'] = [a \g(a'),\g\g'].
  \]
  Here we set for short $\g(a') = \Phi(\g)(a')$.

\item[(iv)] For two morphisms $[a,\g], [a',\g'] \colon \p'\to\p$ and
  $r\in \shr$, the $\shr$-linear structure of $\stka_0$ is given by
  \[
  r[a,\g] = [r a,\g], \quad [a,\g] + [a',\g'] = [a + a'
  \ell_\p(\g'\g^{-1}),\g].
  \]

\item[(v)] The restriction functors are the natural ones.
\end{itemize}

\begin{proposition}\label{pro:P}
  Let $(\stkp,\Phi,\ell)$ be an $\shr$-algebroid data. The description
  (i)--(v) above defines a separated $\shr$-prestack $\stka_0$ on
  $X$. The associated stack $\stka$ is an $\shr$-algebroid endowed
  with a functor $J\colon\stkp \to \stka$ such that
  $\shEnd[\stka](J(\p)) \simeq \Phi(\p)$ for any $\p\in\stkp$.
\end{proposition}

\begin{proof}
  (a) Let us show that the composition is well defined. Consider two
  composable morphisms $[a,\g]\colon \p' \to \p$ and $[a',\g']\colon
  \p'' \to \p'$. At the level of representatives, set $(a,\g) \circ
  (a',\g') = (a \g(a'),\g\g')$.

  \smallskip\noindent (a-i) Let us show that for
  $\phi\in\shEnd[\stkp](\p)$ we have
  \[
  (a,\, \phi \g) \circ (a',\g') \sim (a \ell_\p(\phi),\, \g) \circ
  (a',\g').
  \]
  For this, we have to check that
  \[
  (a \phi(\g(a')),\, \phi \g \g') \sim (a\ell_\p(\phi)\g(a'),\, \g
  \g').
  \]
  This follows from
  \[
  a\ell_\p(\phi) \g(a') = a\phi(\g(a')) \ell_\p(\phi).
  \]

  \smallskip\noindent (a-ii) Similarly, for
  $\phi'\in\shEnd[\stkp](\p')$ we have to prove that
  \[
  (a,\g) \circ (a',\phi'\g') \sim (a,\g) \circ (a'
  \ell_{\p'}(\phi'),\, \g').
  \]
  In other words, we have to check that
  \[
  (a \g(a'),\, \g \phi' \g') \sim (a \g(a' \ell_{\p'}(\phi')),\, \g
  \g' ).
  \]
  This follows from $\g \phi' \g' = (\g \phi' \g^{-1}) \g\g'$ and
  \[
  a \g(a' \ell_{\p'}(\phi')) = a \g(a') \g(\ell_{\p'}(\phi')) = a
  \g(a') \ell_\p(\g\phi'\g^{-1}),
  \]
  where the last equality is due to \eqref{eq:PhiCond}.

  \smallskip\noindent (a-iii) Associativity is easily checked.

  \smallskip\noindent (b) The $\shr$-linear structure is well defined
  by an argument similar to that in part (a) above.

  \smallskip\noindent (c) The functor $J\colon\stkp \to \stka$ is
  induced by the functor $J_0\colon\stkp \to \stka_0$ defined by
  $\p\mapsto\p$ on objects and $\g\mapsto [1,\g]$ on morphisms. The
  morphism $\Phi(\p) \to \shEnd[\stka](J(\p))$, $a\mapsto [a,\id]$ has
  an inverse given by $[a,\g] \mapsto a \ell_\p (\g)$.
\end{proof}

Note that the functor $J\colon\stkp \to \stka$ is neither faithful nor full, in general.

\begin{remark}
  For an $\shr$-algebroid $\stka$, denote by $\stka^\times$ the gerbe
  with the same objects as $\stka$ and isomorphisms as morphisms. Then
  $\stka$ is the $\shr$-algebroid associated with the data
  $(\stka^\times, \Phi_\stka, \ell)$, where $\Phi_\stka(\p) =
  \shEnd[\stka](\p)$ and $\ell_\p$ is the identity.
\end{remark}

\begin{example}
\label{ex:1/2}
Let $X$ be a complex manifold and $\sho_X$ its structure sheaf.
To an invertible $\sho_X$-module $\shl$ one associates an invertible $\Z/2\Z$-gerbe $\stkp_{\shl^{\otimes 1/2}}$ defined as follows.
\begin{itemize}
\item[(i)]
Objects on $U$ are pairs $(\shf,f)$ where $\shf$ is an invertible $\sho_U$-module and $f\colon \shf^{\otimes 2} \isoto \shl$ is an $\sho_U$-linear isomorphism. 
\item[(ii)]
If $(\shf',f')$ is another object, a morphism $(\shf',f')\to(\shf,f)$ is an $\sho_U$-linear isomorphism $\varphi\colon \shf'\isoto \shf$, such that $f' = f\varphi^{\otimes 2}$.
\end{itemize}
Note that any $\psi\in\shEnd[\stkp_{\shl^{\otimes 1/2}}]\bigl((\shf,f)\bigr)$ is a locally constant $\Z/2\Z$-valued function.
Denote by $\C_{\shl^{\otimes 1/2}}$ the invertible $\C$-algebroid associated with the data $(\stkp_{\shl^{\otimes 1/2}},\Phi,\ell)$, where $\Phi\bigl((\shf,f)\bigr)=\C_U$, $\Phi(\varphi)=\id$, $\ell_{(\shf,f)}(\psi) = \psi$.
\end{example}

\section{Contactification of symplectic manifolds}\label{se:geo}

We first review here some notions from contact and symplectic
geometry. In particular, we discuss the gerbe parameterizing the
primitives of the symplectic $2$-form. Then, we show how any
Lagrangian subvariety of a complex symplectic manifold can be uniquely
lifted to a local contactification.

\subsection{The gerbe of primitives}

Let $X$ be a complex manifold and $\sho_X$ its structure sheaf. Denote
by $TX$ and $T^*X$ the tangent and cotangent bundle, respectively, and
by $\Theta_X$ and $\Omega^{1}_X$ their sheaves of sections. For $k\in\Z$
denote by $\Omega_X^{k}$ the sheaf of holomorphic $k$-forms. For
$v\in\Theta_X$ denote by $i_v\colon \Omega_X^{k} \to \Omega_X^{k-1}$ the
inner derivative and by $L_v\colon \Omega_X^{k} \to \Omega_X^{k}$ the Lie
derivative.

\medskip

Let $\omega\in\sect(X;\Omega_X^{2})$ be a $2$-form which is closed,
i.e.\ $d\omega=0$.

\begin{definition}
  The gerbe $\stkc'_\omega$ on $X$ is the stack associated with the
  separated prestack defined as follows.
  \begin{enumerate}
  \item Objects on $U\subset X$ are primitives of $\omega|_U$, i.e.\
    $1$-forms $\theta\in\sect(U;\Omega^{1}_X)$ such that $d \theta =
    \omega|_U$.

  \item If $\theta'$ is another object, a morphism $\theta' \to
    \theta$ is a function $\varphi\in\sect(U;\sho_X)$ such that
    $d\varphi = \theta' - \theta$.  Composition with $\varphi'\colon
    \theta'' \to \theta'$ is given by $\varphi\circ\varphi' = \varphi
    + \varphi'$.
  \end{enumerate}
\end{definition}

The following result is clear.

\begin{lemma}\label{lem:Cw1w2}
  \begin{itemize}
  \item [(i)] The stack $\stkc'_\omega$ is an invertible $\C$-gerbe.

  \item [(ii)] If $\omega'\in\Omega_X^{2}(X)$ is another closed
    $2$-form, there is an equivalence
    \[
    \stkc'_{\omega}\mathop{\times}\limits^{\Ga}\stkc'_{\omega'} \isoto
    \stkc'_{\omega+\omega'}.
    \]
  \end{itemize}
\end{lemma}

Here, for a commutative sheaf of groups $\shg$,
$\stkp\mathop{\times}\limits^{\shg}\stkq$ denotes the contracted
product of two $\shg$-gerbes. This is the stack associated to the
prestack whose objects are pairs $(p,q)$ of an object of $\stkp$ and
an object of $\stkq$, with morphisms
\[
\shHom[\stkp\mathop{\times}\limits^{\shg}\stkq]\bigl((p,q),(p',q')\bigr)
= \shHom[\stkp](p,p')
\mathop{\times}\limits^{\shg}\shHom[\stkq](q,q').
\]

For a principal $\C$-bundle $\rho\colon Y\to X$, denote
\[
T_\lambda\colon Y \to Y, \qquad \infGa = \left.\textstyle{\frac
d{d\lambda}}T_\lambda\right|_{\lambda = 0}\in\Theta_Y
\]
the action of $\lambda\in\C$ and the infinitesimal generator of the
$\Ga$-action, respectively.

\begin{definition}
\label{def:Comega}
  The gerbe $\stkc_\omega$ on $X$ is defined as
  follows. \begin{enumerate} \item Objects on $U\subset X$ are pairs
    $\rho = (V\to[\rho]U,\alpha)$ of a principal $\C$-bundle $\rho$
    and a $1$-form $\alpha\in\sect(V;\Omega^{1}_V)$ such that
    $i_{\infGa} \alpha = 1$ and $\rho^*\omega = d\alpha$. In
    particular, $L_{\infGa}\alpha = 0$.

  \item For another object $\rho' = (V'\to[\rho']U,\alpha')$,
    morphisms $\chi\colon\rho' \to \rho$ are morphisms of principal
    $\Ga$-bundles such that $\chi^*\alpha = \alpha'$.
  \end{enumerate}
\end{definition}

Denote by $p_1\colon X\times\C\to[p_1] X$ the trivial principal $\C$-bundle given by 
the first projection. Let $t$ be the coordinate of $\C$. For a primitive $\theta$ of $\omega$, an object of $\stkc_\omega$ is given by $(p_1,p_1^*\theta+dt)$.
By the next lemma, any object $\rho$ of $\stkc_\omega$ is locally of this form and
any automorphism of
$\rho$ is locally of the form $T_\lambda$, for $\lambda\in\C$.  
(See~\cite[Remark~9.3]{PS04} for similar
observations.)

\begin{lemma}\label{lem:cont}
  There is a natural equivalence $\stkc'_\omega \isoto
  \stkc_\omega$. In particular, $\stkc_\omega$ is an invertible
  $\Ga$-gerbe.
\end{lemma}

\begin{proof}
  As above, denote by $p_1\colon X\times\C\to[p_1] X$ the first projection and
  by $t$ the coordinate of $\C$. Consider the functor $B\colon
  \stkc'_\omega \to \stkc_\omega$ given by $\theta\mapsto
  (p_1,p_1^*\theta+dt)$ on objects and $\varphi\mapsto \bigl( (x,t)
  \mapsto (x,t+\varphi(x)) \bigr)$ on morphisms.

  As $B$ is clearly faithful, we are left to prove that it is locally
  full and locally essentially surjective. For the latter, let $\rho =
  (V\to[\rho]U,\alpha)$ be an object of $\stkc_\omega(U)$. Up to
  shrinking $U$, we may assume that the bundle $\rho$ is
  trivial. Choose an isomorphism of principal $\Ga$-bundles $\xi\colon
  U\times \Ga \to V$. As $i_{\infGa}(\xi^*\alpha -dt) =
  L_{\infGa}(\xi^*\alpha -dt) = 0$, there exists a unique $1$-form
  $\theta\in\Omega^{1}_X(U)$ such that $\xi^*\alpha -dt =
  p_1^*\theta$. Then $\omega|_U=d\theta$ and $\rho \simeq B(\theta)$.

  It remains to show that any morphism $\chi\colon\rho'\to\rho$ of
  $\stkc_\omega(U)$ is in the image of $B$. Up to shrinking $U$, we
  may assume that $\rho=(p_1,p_1^*\theta+dt)$ and
  $\rho'=(p_1,p_1^*\theta'+dt)$. Then $\chi\colon X\times\C \to
  X\times\C$ is given by $(x,t) \mapsto (x,t+\varphi(x))$ for some
  $\varphi\in\sho_X(U)$. Since $\chi^*(p_1^*\theta+dt) =
  p_1^*\theta'+dt$, it follows that $d\varphi = \theta' -
  \theta$. Hence $\chi = B(\varphi)$.
\end{proof}

\medskip

Let $R$ be a commutative ring endowed with a group homomorphism
$\ell\colon\Ga \to R^\times$.

\begin{definition}\label{def:Rw}
  The stack $R_\omega$ is the invertible $R$-algebroid associated with
  the data $(\stkc_\omega,\Phi_R,\ell)$, where
  \begin{align*}
    \Phi_R(\rho) = R_U, \quad \Phi_R(\chi) = \id_{R_U}, \quad
    \ell_\rho(T_\lambda) = \ell(\lambda),
  \end{align*}
  for $\rho = (V\to[\rho]U,\alpha)$, $\chi\colon \rho'\to\rho$ and
  $\lambda\in\C$.
\end{definition}

Note that by Lemma~\ref{lem:Cw1w2} there is an $R$-linear equivalence
\[
R_{\omega} \tens[R_X] R_{\omega'} \isoto R_{\omega+\omega'}.
\]

\begin{remark}
  Equivalence classes of invertible $\Ga$-gerbes and of invertible
  $R$-algebroids are classified by $H^2(X;\Ga)$ and $H^2(X;R^\times)$,
  respectively. The class of $\stkc_\omega$ coincides with the de~Rham
  class $[\omega]$ of the closed $2$-form $\omega$, and the class of
  $R_\omega$ is the image of $[\omega]$ by
  $\ell\colon H^2(X;\Ga) \to H^2(X;R^\times)$.
\end{remark}

\subsection{Symplectic manifolds}

A complex symplectic manifold $X=(X,\omega)$ is a complex
manifold $X$ of even dimension endowed with a holomorphic closed $2$-form
$\omega\in\sect(X;\Omega^{2}_X)$ which is non-degenerate, i.e.\ the
$n$-fold exterior product $\omega\wedge\cdots\wedge\omega$ never
vanishes for $n = \frac12 \dim X$.

Let $H\colon \Omega_X^{1} \isoto \Theta_X$ be the Hamiltonian
isomorphism induced by the symplectic form $\omega$. The Lie bracket
of $\varphi,\varphi'\in \sho_X$ is given by $\{\varphi,\varphi'\} =
H_\varphi(\varphi')$, where $H_\varphi = H(d\varphi)$ is the
Hamiltonian vector field of $\varphi$.

\begin{example}\label{ex:T*M}
  Let $M$ be a complex manifold. Its cotangent bundle $T^*M$ has a natural
  symplectic structure $(T^*M,d\theta)$, where $\theta$ denotes
  the canonical $1$-form. Let $(x)=(x_1,\dots,x_n)$ be a system of
  local coordinates on $M$.  The associated system $(x;u)$ of local
  symplectic coordinates on $T^*M$ is given by $p = \sum_i u_i(p)
  dx_i$. Then the canonical $1$-form is written $\theta=\sum_i u_i
  dx_i$ and the Hamiltonian vector field of $\varphi\in\sho_M$ is
  written $H_\varphi = \sum_i \bigl( \varphi_{u_i} \partial_{x_i} -
  \varphi_{x_i}
  \partial_{u_i} \bigr)$.
\end{example}

An analytic subset $\Lambda\subset X$ is called involutive if for
any $f,g\in\sho_X$ with $f|_\Lambda = g|_\Lambda = 0$ one has
$\{f,g\}|_\Lambda = 0$. The analytic subset $\Lambda$ is called
Lagrangian if it is involutive and $\dim X = 2 \dim\Lambda$.

Let $X' = (X',\omega')$ be another symplectic manifold. A symplectic
transformation $\psi\colon X'\to X$ is a holomorphic isomorphism such
that $\psi^*\omega = \omega'$.

By Darboux
theorem, for any complex symplectic manifold $X$ there locally exist 
symplectic transformations
\begin{equation}
  \label{eq:DarbS}
  X \supset U\to[\psi] U_M\subset T^*M,
\end{equation}
for a complex manifold $M$  with $\dim M = \frac12\dim X$.

\subsection{Contact manifolds}

Let $\gamma\colon Z\to Y$ be a principal $\Gm$-bundle over a complex
manifold $Y$. Denote by $\infGm$ the infinitesimal generator of the
$\Gm$-action on $Z$. For $k\in\Z$, let $\sho_Z(k)$ be the sheaf of
$k$-homogeneous functions, i.e.~solutions $\varphi\in\sho_Z$ of
$\infGm \varphi = k\varphi$. Let $\sho_Y(k) = \oim\gamma\sho_Z(k)$ be
the corresponding invertible $\sho_Y$-module, so that $\sho_Y(-1)$ is
the sheaf of sections of the line bundle
$\C\mathop{\times}\nolimits^{\Gm}Z$.

A complex contact manifold $Y = (Z\to[\gamma]Y,\theta)$ is a complex
manifold $Y$ endowed with a principal $\Gm$-bundle $\gamma$ and a
holomorphic $1$-form $\theta\in\sect(Z;\Omega^{1}_Z)$ such that $(Z,
d\theta)$ is a complex symplectic manifold, $i_{\infGm} \theta = 0$
and $L_{\infGm} \theta = \theta$, i.e.~$\theta$ is $1$-homogeneous.

\begin{example}
  Let $M$ be a complex manifold and $\theta$ the canonical $1$-form on
  $T^*M$ as in Example~\ref{ex:T*M}. The projective cotangent bundle
  $P^*M$ has a natural contact structure $(\gamma,\theta)$ with
  $\gamma\colon T^*M \setminus M \to P^*M$ the projection. Here
  $T^*M\setminus M$ denotes the cotangent bundle with the zero-section
  removed.
\end{example}

Note that the $1$-form $\theta$ on $Z$ may be considered as a global section of
$\Omega_Y^{1}\tens[\sho]\sho_Y(1)$. In particular, there is an embedding
\begin{equation}\label{eq:OY-1}
  \iota\colon\sho_Y(-1) \to \Omega_Y^{1}, \quad \varphi\mapsto \varphi\theta.
\end{equation}
Note also that the symplectic manifold $Z$ is homogeneous with respect
to the $\Gm$-action, i.e.~$\theta = i_{\infGm}(d \theta)$.  Moreover,
there exists a unique $\Gm$-equivariant embedding $Z \hookrightarrow
T^*Y$ such that $\theta$ is the pull-back of the canonical $1$-form on
$T^*Y$.

Since $d\theta$ is $1$-homogeneous, the Hamiltonian vector field
$H_\varphi$ of $\varphi\in \sho_Z(k)$ is $(k-1)$-homogeneous,
i.e.~$[\infGm,H_\varphi] = (k-1)\, H_\varphi$.

An analytic subset $\Gamma$ of $Y$ is called involutive (resp.\ Lagrangian) 
if $\opb\gamma\Gamma$ is involutive (resp.\ Lagrangian) 
in $Z$.

Let $Y' = (Z'\to[\gamma']Y',\theta')$ be another contact manifold. A
contact transformation $\chi\colon Y'\to Y$ is an isomorphism of
principal $\Gm$-bundles
\[
\xymatrix{
  Z' \ar[r]^{\widetilde\chi} \ar[d]^{\gamma'} & Z \ar[d]^\gamma \\
  Y' \ar[r]^\chi & Y }
\]
such that $\widetilde\chi^*\theta = \theta'$.

By Darboux
theorem, for any complex contact manifold $Y$ there locally exist contact transformations
\begin{equation}
  \label{eq:DarbC}
  Y \supset V\to[\chi]  V_M\subset P^*M,
\end{equation}
for a complex manifold $M$ with $\dim M = \frac12(\dim Y + 1)$.

\subsection{Contactifications}\label{sse:contsymp}

Let $X = (X,\omega)$ be a complex symplectic manifold.
A contactification of $X$
is a global object of the stack $\stkc_\omega$ described in Definition~\ref{def:Comega}.
Morphisms of contactifications are morphisms in
$\stkc_\omega$.

For a contactification $\rho=(Y\to[\rho]X,\alpha)$ of $X$,
the total space $Y$ of $\rho$ has a natural complex contact structure given by
$(Y\times\C^\times\to[q_1] Y,\tau\,q_1^*\alpha)$, where $q_1$ is the first
projection and $\tau\in\C^\times$.
Note that, in terms of contact structures, a morphism $\rho'\to\rho$ of contactifications is a contact transformation $\chi\colon Y'\to Y$ over $X$.

\begin{example}\label{ex:PMC}
Let $M$ be a complex manifold and denote by $(t;\tau)$ the
symplectic coordinates of $T^*\C$. Consider the principal $\Ga$-bundle
\[
P^*(M\times\C) \supset \{\tau\neq 0\}\to[\rho] T^*M,
\quad (x,t;\xi,\tau) \mapsto
(x;\xi/\tau),
\]
with the $\Ga$-action given by
translation in the $t$ variable.  
Note that
the bundle $\rho$ is trivialized by
\[
\chi\colon \{\tau\neq 0\} \isoto (T^*M) \times \C, \quad
(x,t;\xi,\tau) \mapsto ((x;\xi/\tau),t).
\]
Consider the projection $p_1\colon (T^*M) \times \C \to T^*M$.

As in Example~\ref{ex:T*M}, denote by $\theta$ the canonical $1$-form of $T^*M$. Then a contactification of
$(T^*M,d\theta)$ is given by $(\rho,\alpha)$, with $\rho$ as above and $\alpha = \chi^*(p_1^*\theta+dt)$. In a system $(x;u)$ of local symplectic coordinates on $T^*X$, one has $\theta = u\,dx$ and $\alpha = (\xi/\tau)dx + dt$.  
As the canonical $1$-form of $T^*(M\times\C)$ is $\tau\alpha = \xi\,dx+\tau\,dt$, the map \eqref{eq:OY-1} is given by
\[
\iota\colon\sho_{P^*(M\times\C)}(-1)|_{\{\tau\neq 0\}} \to \Omega_{P^*(M\times\C)}^{1}|_{\{\tau\neq 0\}}, \quad \varphi\mapsto \varphi\,\tau\alpha.
\]
\end{example}

\subsection{Contactification of Lagrangian subvarieties}

In this section we show how any Lagrangian subvariety of a complex
symplectic manifold lifts to a contactification (see
e.g.~\cite[Lemma~8.4]{DS07} for the case of Lagrangian submanifolds).

Let us begin with a preliminary lemma.

\begin{lemma}\label{lem:f}
  Let $M$ be a complex manifold, $S\subset M$ a closed analytic subset
  and $\theta\in\Omega^{1}_M$ a $1$-form such that
  $d\theta|_{S_\reg}=0$. Then there locally exists a continuous function $f$,
  on $S$ such that $f$ is holomorphic on the non-singular locus $S_\reg$,
  and $df|_{S_\reg} = \theta|_{S_\reg}$.
\end{lemma}

\begin{proof}
  Let $S'\to S$ be a resolution of singularities and let $p\colon S'
  \to M$ be the composite $S'\to S \hookrightarrow M$. Thus $S'$ is a
  complex manifold, $p$ is proper and $\opb p (S_\reg) \to S_\reg$ is
  an isomorphism. Consider the global section $\theta' = p^*\theta$ of
  $\Omega_{S'}^{1}$. As $d\theta|_{S_\reg}=0$ and $\opb p (S_\reg)$ is
  dense in $S'$, we have $d\theta'=0$.

  Fix a point $s_0\in S$ and set $S'_0 = \opb p(s_0)$.
Since $\theta'|_{(S'_0)_\reg} = 0$, there exists a unique 
holomorphic function $f'$ defined on a neighborhood of $S'_0$
such that $df' =  \theta'$ and $f'\vert_{S'_0}=0$. As $p$ is proper, replacing $M$ by a neighborhood of $s_0$
  we may assume that $f'$ is globally defined on $S'$.

  Set $S''=S'\times_S S'$ and $S''_0 = S'_0\times_S S'_0$. We may
  assume that $S''_0$ intersects each connected component of
  $S''$. Consider the diagram
  \[
  \xymatrix{ S''_\reg \ar[r]^q & S'' \ar@<-.7ex>[r]_{p_2}
    \ar@<+.7ex>[r]^{p_1} & S' \ar[r]^-p & M, }
  \]
  where $p_1$ and $p_2$ are the projections $S'\times_S S' \to S'$. To
  conclude, it is enough to prove that $g= p_1^*f' - p_2^*f'$
  vanishes, for then we can set $f(w) = f'(w')$ with $p(w') = w$.

  Since $pp_1 = pp_2$, one has $d q^*g = d(pp_1q)^*\theta -
  d(pp_2q)^*\theta = 0$ so that $g$ is locally constant on
  $S''_\reg$. Hence $g$ is locally constant by
  Sublemma~\ref{lem:loccst} below with $T=S''$ and $U=S''_\reg$.
  Since $g$ vanishes on $S''_0$, it vanishes everywhere.
\end{proof}

\begin{sublemma}\label{lem:loccst}
  Let $T$ be a Hausdorff topological space and $U\subset T$ a dense
  open subset.  Assume there exists a basis $\shb$ of open subsets of
  $T$ such that any $B\in\shb$ is connected and $B\cap U$ has finitely
  many connected components.  If a continuous function on $T$ is
  locally constant on $U$, then it is locally constant on $T$.
\end{sublemma}

Let now $X = (X,\omega)$ be a complex symplectic manifold.

\begin{proposition}\label{pro:Lambda}
  Let $\Lambda$ be a Lagrangian subvariety of $X$. Then there exist a
  neighborhood $U$ of $\Lambda$ in $X$ and a pair $(\rho,\Gamma)$ with
  $\rho\colon V\to U$ a contactification and $\Gamma$ a Lagrangian
  subvariety of $V$ such that $\rho|_\Gamma$ is a homeomorphism over
  $\Lambda$ and a holomorphic isomorphism over $\Lambda_\reg$.
\end{proposition}

\begin{proof}
  Let $\{U_i\}_{i\in I}$ be an open cover of $\Lambda$ in $X$ such
  that for each $i\in I$ there is a primitive
  $\theta_i\in\Omega_X^{1}(U_i)$ of $\omega|_{U_i}$.  Set $\Lambda_i =
  \Lambda \cap U_i$. Using Lemma~\ref{lem:f}, up to shrinking the
  cover we may assume that there is a continuous function $f_i$ on
  $\Lambda_i$ such that $f_i|_{\Lambda_{i,\reg}}$ is a primitive of
  $\theta_i|_{\Lambda_{i,\reg}}$. Set $U_{ij} = U_i\cap U_j$ and
  similarly for $\Lambda_{ij}$. Up to further shrinking the cover we
  may assume that $\Lambda_{ij}$ intersects each connected component
  of $U_{ij}$ and there is a function $\varphi_{ij}\in \sho_X(U_{ij})$
  such that $d\varphi_{ij} = \theta_i - \theta_j|_{U_{ij}}$ and
  $\varphi_{ij}|_{\Lambda_{ij,\reg}} = f_i -
  f_j|_{\Lambda_{ij,\reg}}$. Set $U_{ijk} = U_i\cap U_j\cap U_k$ and
  similarly for $\Lambda_{ijk}$. Note that $d(\varphi_{ij}
  +\varphi_{jk} +\varphi_{ki}) = 0$, so that $\varphi_{ij}
  +\varphi_{jk} +\varphi_{ki}$ is locally constant on $U_{ijk}$. Since
  it vanishes on $\Lambda_{ijk}$, it vanishes everywhere.

  Set $\rho_i = (V_i\to[p_1]U_i,\alpha_i)$, where $V_i = U_i \times
  \C$ and $\alpha_i = p_1^*\theta_i + dt$. Let $(\rho_i,\Gamma_i)$ be
  the pair with
  \[
  \Gamma_i = \{(x,t)\in V_i \setdef x\in \Lambda_i,\ t+f_i(x) = 0\}.
  \]
  Then the pair $(\rho,\Gamma)$ is obtained by patching the
  $(\rho_i,\Gamma_i)$'s via the maps $(x,t) \mapsto
  (x,t+\varphi_{ij}(x))$.
\end{proof}

Let us give an example that shows how, in general, $\Gamma$ and
$\Lambda$ are not isomorphic as complex spaces.

\begin{example}
  Let $X=(T^*\C,d\theta)$ with symplectic coordinates $(x;u)$, and $Y=
  (X \times \C,\alpha)$ with extra coordinate $t$. Then $\theta = u\,
  dx$ and $\alpha = u\, dx+dt$. Take as $\Lambda\subset X$ a
  parametric curve $\Lambda = \{(x(s),u(s))\setdef s\in\C \}$, with
  $x(0) = u(0) = 0$. Then
  \[
  \Gamma = \{(x,u,t)\setdef x=x(s),\ u=u(s),\ t+f(s)=0 \},
  \]
  where $f$ satisfies the equations $f'(s) = u(s)x'(s)$ and
  $f(0)=0$. For
  \[
  x(s) = s^3,\quad u(s) = s^7 + s^8,\quad f(s) = \tfrac3{10}s^{10} +
  \tfrac3{11}s^{11},
  \]
  we have an example where $f$ cannot be written as an analytic
  function of $(x,u)$. In fact, $s^{11} =11x(s)u(s) -
  \frac{110}3 f(s)$ and $s^{11} \notin
  \C[\mspace{-1mu}[s^3,s^7+s^8]\mspace{-1mu}]$.
\end{example}

\section{Holonomic modules on symplectic manifolds}\label{se:contact}

We start by giving here a construction of the microdifferential
algebroid of~\cite{Kas96} in terms of algebroid data and by recalling
some results on regular holonomic microdifferential modules. Then,
using the results from the previous section, we show how it is
possible to associate to a complex symplectic manifold a natural
$\C$-linear category of holonomic modules.

\subsection{Microdifferential algebras}\label{sse:EM}

Let us review some notions from the theory of microdifferential
operators (refer to \cite{S-K-K,Kas03}).

Let $M$ be a complex manifold. Denote by $\she_M$ the sheaf on $P^*M$
of microdifferential operators, and by $\filt[k]{\she_M}$ its subsheaf of operators of order at most
$k\in\Z$. Then $\she_M$ is a sheaf of $\C$-algebras on $P^*M$, filtered over $\Z$ by
the $\filt[k]{\she_M}$'s. 

Take a local symplectic coordinate system $(x;\xi)$ on $T^*M$. For an open subset $U\subset T^*M$, a section
$a\in\sect(U;\filt[k]{\she_M})$ is represented by its total symbol, which is
a formal series
\[
a(x,\xi)=\sum_{j\leq k} a_j(x,\xi), \qquad a_j\in\sect(U;\sho_{P^*M}(j))
\]
satisfying suitable growth conditions. 
In terms of total symbols, the product in $\she_M$ is given by Leibniz rule.
More precisely, for
$a'\in\she_M$ with total symbol $a'(x,\xi)$, the product $aa'$ has
total symbol
\[
\sum_{\multiindex\in \N^n} \frac{1}{\multiindex
    !} \partial^{\multiindex}_\xi a(x,\xi)
  \partial^{\multiindex}_x a'(x,\xi).
\]

For $a\in\filt[k]{\she_M}$, the top degree component
$a_k\in\sho_{P^*M}(k)$ of its total symbol does not depend of the choice of coordinates.
The map
\[
\sigma_k \colon \filt[k]{\she_M} \to
\sho_{P^*M}(k), \quad a\mapsto a_k
\]
induced by the isomorphism $\filt[k]{\she_M}/\filt[k-1]{\she_M} \simeq
\sho_{P^*M}(k)$ is called the symbol map. 
Recall that an operator $a\in\filt[k]{\she_M}\setminus\filt[k-1]{\she_M}$ is invertible at $p\in P^*M$ if and only if $\sigma_k(a)(p) \neq 0$.

For $a\in\filt[k]{\she_M}$
and $a'\in\filt[k']{\she_M}$, one has
\[
\{\sigma_k(a), \sigma_{k'}(a')\} = \sigma_{k+k'-1}([a,a']).
\]

An anti-involution of $\she_M$ is an
isomorphism of $\C$-algebras $*\colon \she_M\to\she_M^\op$ such that
$** = \id$.

\begin{remark}\label{rem:Eloc}
   In a local system of symplectic coordinates, an example of
  anti-involution $*$ of $\she_M$ is given by the formal adjoint.
This is  described at the level of total symbols by
\[
a^*(x,\xi) =
  \sum_{\multiindex\in
    \N^n}\frac{1}{\multiindex!}\partial_\xi^\multiindex\partial_x^\multiindex
  \bigl(a(x,-\xi)\bigr). 
\]
The formal adjoint depends on the choice
  of the top-degree form $dx_1\wedge\cdots\wedge dx_n$.
\end{remark}

Consider a contact transformation
\[
P^*M' \supset V' \to[\chi] V \subset P^*M
\]
where $M,M'$ are complex manifolds with the same dimension.
It is a fundamental result of \cite{S-K-K} that quantized contact transformations
can be locally quantized. 

\begin{theorem}
\label{thm:qchi}
With the above notations:
\begin{itemize}
\item[(i)]
Any $\C$-algebra isomorphism $f\colon \chi_*\she_{M'}|_V \isoto \she_M|_V$ is a filtered isomorphism, and $\sigma_k(f(a')) = \chi_*\sigma_k(a')$ for any
$a'\in\filt[k]{\she_{M'}}$.
\item[(ii)]
For any $p\in V$ there exists a neighborhood $U$ of $p$ in $V$
and a $\C$-algebra isomorphism
$f\colon \chi_*\she_{M'}|_U \isoto \she_M|_U$.
\item[(iii)]
Let $*$ and $*'$ be anti-involutions of $\she_M|_V$ and $\she_{M'}|_{V'}$, respectively.
For any $p\in V$ there exists a neighborhood $U$ of $p$ in $V$
and a $\C$-algebra isomorphism $f$ as in (ii) such that $f*' = * f$.
\end{itemize}
\end{theorem}

An isomorphism $f$ as in (ii)
is called a quantized
contact transformation over $\chi$.
Quantized contact transformations over $\chi$ are not unique.
It was noticed
in~\cite{Kas96} that one can reduce the ambiguity to an inner
automorphism by considering anti-involutions as in (iii) (see Lemma~\ref{lem:stkpe} below).

The $\C$-algebra $\she_M$ is left and right Noetherian. It is another fundamental result of \cite{S-K-K} that the support of a coherent $\she_M$-module
is a closed involutive subvariety of $P^*M$. A coherent $\she_M$-module supported by a Lagrangian subvariety is called holonomic. We refer e.g.\ to \cite{Kas03} for the notion of regular holonomic $\she_M$-module.

\subsection{Microdifferential algebroid}\label{sse:E}

Let $Y$ be a complex contact manifold.

\begin{definition}\label{def:Ealg}
  A microdifferential algebra $\she$ on $Y$ is a sheaf of
  $\C$-algebras such that, locally on $Y$, there is a $\C$-algebra
  isomorphism $\she|_V \simeq \opb\chi\she_M$ in a Darboux chart
  \eqref{eq:DarbC}.
\end{definition}

Since any $\C$-algebra automorphism of $\she_M$ is filtered and
symbol preserving, it follows that a microdifferential algebra $\she$
on $Y$ is filtered and has symbol maps
\[
\sigma_k \colon \filt[k]\she \to \sho_Y(k).
\]

\begin{example}\label{ex:eiii}
  Let $Y=P^*M$ be the projective cotangent bundle to a complex
  manifold $M$ and denote by $\Omega_M = \Omega_M^{\dim M}$ the invertible $\sho_M$-module
  of top-degree forms. Consider the algebra of twisted microdifferential operators
  \[
  \she_{\Omega_M^{\otimes 1/2}} =
  \Omega_M^{\otimes 1/2}\tens[\sho_M]\she_M\tens[\sho_M]\Omega_M^{\otimes-1/2}.
  \]
Then $\she_{\Omega_M^{\otimes 1/2}}$ is a microdifferential algebra on $P^*M$,
and the formal adjoint $*$ of Remark~\ref{rem:Eloc}
  gives a canonical anti-involution of $\she_{\Omega_M^{\otimes 1/2}}$.
\end{example}

\begin{definition}\label{def:Pa}
  The gerbe $\stkpe$ on $Y$ is defined as follows.
  \begin{enumerate}
  \item For an open subset $V\subset Y$, objects of $\stkpe(V)$ are
    pairs $\p = (\she,\,*)$ of a microdifferential algebra $\she$ on
    $V$ and an anti-involution $*$ of $\she$.

  \item If $\p' = (\she',*')$ is another object,
    \[
    \shHom[\stkpe](\p',\p) = \{f\in \shIso[\C\text{-}\stkalg](\she',
    \she) \setdef f *' = * f \}.
    \]
  \end{enumerate}
\end{definition}

(The fact that the stack of groupoids $\stkpe$ is a gerbe follows from Theorem~\ref{thm:qchi}.)

\begin{lemma}[{\cite[Lemma~1]{Kas96}}]\label{lem:stkpe}
   For any $\p = (\she,\,*)\in\stkpe$ there is
  an isomorphism of sheaves of groups
  \[
  \psi \colon \{b \in \she^\times \setdef b^*b=1,\
  \sigma_0(b)=1 \} \isoto \shEnd[\stkpe](\p), \quad b\mapsto \Ad(b).
  \]
\end{lemma}

By this lemma, we have a natural $\C$-algebroid data on $Y$, and hence
a $\C$-algebroid.

\begin{definition}\label{def:E}
  The microdifferential algebroid $\stke_Y$ is the $\C$-algebroid
  associated to $(\stkpe,\Phi_\stke,\ell)$ where
  \[
  \Phi_\stke(\p) = \she, \quad \Phi_\stke(f) = f, \quad \ell_\p(g) =
  b,
  \]
  for $\p = (\she,\,*)$, $f\colon \p'\to\p$ and $g = \psi(b)$.
\end{definition}

By the construction in \S~\ref{sse:algdata}, this means that objects of $\stke_Y$ are microdifferential algebras $(\she,*)$ endowed with an anti-involution.
Morphisms $(\she',*')\to(\she,*)$ in $\stke_Y$ are equivalence classes of pairs
$(a,f)$ with $a\in\she$ and $f\colon\she'\isoto\she$ such that $f*' =
*f$. The equivalence relation is given by $(a,\Ad(b)f) \sim (ab,f)$ for $b \in \she^\times$ with
$b^*b=1$ and $\sigma_0(b)=1$.

\begin{remark}\label{re:eiii}
  Let $Y=P^*M$ be the projective cotangent bundle to a complex
  manifold $M$. With notations as in Example~\ref{ex:eiii}, a  global object of
  $\stke_{P^*M}$ is given by $(\she_{\Omega_M^{\otimes 1/2}},*)$. 
This implies that the algebroid $\stke_{P^*M}$ is represented by the
  microdifferential algebra $\she_{\Omega_M^{\otimes 1/2}}$.
\end{remark}

\subsection{Holonomic modules on contact manifolds}\label{se:regE}

Let $Y = (Z\to[\gamma]Y,\theta)$ be a complex contact manifold.
Consider the stack $\stkMod(\stke_Y)$ of modules over the microdifferential algebroid $\stke_Y$.
For a subset $S \subset Y$, denote by
$\stkMod[S](\stke_Y)$ the full substack of $\stkMod(\stke_Y)$ of objects supported on
$S$.
By construction, $\stke_Y$ is locally represented by microdifferential
algebras.
As the notions of coherent and regular holonomic microdifferential modules are local and invariant by quantized contact transformations,
they make sense also for objects of $\stkMod(\stke_Y)$.
Denote by $\stkMod[\coh](\stke_Y)$ and $\stkMod[\reghol](\stke_Y)$ the full substacks
of $\stkMod(\stke_Y)$ whose objects are coherent and regular holonomic, respectively. 

Let $\stkr$ be an invertible $\C$-algebroid $\stkr$.
Then $\stkMod(\stkr)$ is locally equivalent to $\stkMod(\C_Y)$.
Hence the notion of local system makes sense for objects of $\stkMod(\stkr)$.
Denote by $\stack{LocSys}(\stkr)$ the full
substack of $\stkMod(\stkr)$ whose objects are local systems.

Consider the invertible $\C$-algebroid $\C_{\Omega_\Lambda^{\otimes 1/2}}$ on
$\Lambda$ as in Example~\ref{ex:1/2}.

By~\cite[Proposition~4]{Kas96} (see also \cite[Corollary~6.4]{DS07}),
one has

\begin{proposition}\label{pro:LYsmooth}
  For a smooth Lagrangian submanifold $\Lambda \subset Y$ there is an
  equivalence
  \[
  \stkMod[\Lambda,\reghol](\stke_Y) \simeq
  \oim{p}\stack{LocSys}(\opb{p}\C_{\Omega_\Lambda^{\otimes 1/2}}),
  \]
  where $p\colon \opb\gamma\Lambda\to\Lambda$ is the
  restriction of $\gamma\colon Z\to Y$.
\end{proposition}

Recall that a $\C$-linear triangulated category $\stkt$ is called
Calabi-Yau of dimension $d$ if for each $M,N\in\stkt$ the vector
spaces $\Hom[\stkt](M,N)$ are finite-dimensional and there are
isomorphisms
\[
\Hom[\stkt](M,N)^\vee \simeq \Hom[\stkt](N,M[d]),
\]
functorial in $M$ and $N$.
Here $H^\vee$ denotes the dual of a vector space $H$.

Denote by $\BDC_\reghol(\stke_Y)$ the full triangulated subcategory of
the bounded derived category of $\stke_Y$-modules whose objects have
regular holonomic cohomologies.

The following theorem is obtained in \cite{KS08}\footnote{The statement in \cite[Theorem~9.2~(ii)]{KS08} is not correct. It should be read as Theorem~\ref{th:CY} in the present paper} as a corollary of
results from \cite{KK81}.

\begin{theorem}\label{th:CY}
  If $Y$ is compact, then $\BDC_\reghol(\stke_Y)$ is a $\C$-linear
  Calabi-Yau triangulated category of the same dimension as $Y$.
\end{theorem}

\subsection{Holonomic modules on symplectic manifolds}

Let $X= (X,\omega)$ be a complex symplectic manifold and
$\Lambda\subset X$ a closed Lagrangian subvariety. By
Proposition~\ref{pro:Lambda} there exists a neighborhood $U\supset
\Lambda$, a contactification $\rho\colon V\to U$ and a closed Lagrangian
subvariety $\Gamma \subset V$ such that $\rho$ induces an isomorphism
$\Gamma \to \Lambda$. Let us still denote by $\rho$ the composition
$V\to U\to X$.  We set
\begin{align*}
  \stack{RH}_{X,\Lambda} &= \oim\rho\stkMod[\Gamma,\reghol](\stke_V),\\
  \mathbf{DRH}_\Lambda(X) &= \BDC_{\Gamma,\reghol}(\stke_V).
\end{align*}
By unicity of the pair $(\rho,\Gamma)$, the stack $\stack{RH}_{X,\Lambda}$ and the triangulated category $\mathbf{DRH}_\Lambda(X)$ only depend on $\Lambda$.

For $\Lambda\subset\Lambda'$, there are natural fully faithful, exact functors
\[
\stack{RH}_{X,\Lambda} \to \stack{RH}_{X,\Lambda'}, \quad
\mathbf{DRH}_\Lambda(X) \to \mathbf{DRH}_{\Lambda'}(X).
\]
The family of closed Lagrangian subvarieties of $X$, ordered by inclusion, is filtrant.

\begin{definition}\label{def:hol}
  \begin{itemize}
  \item[(i)] The stack of regular holonomic microdifferential modules
    on $X$ is the $\C$-linear abelian stack defined by
    \[
    \stack{RH}_X = \ilim[\Lambda]\stack{RH}_{X,\Lambda}.
    \]
  \item[(ii)] The triangulated category of complexes of regular
    holonomic microdifferential modules on $X$ is the $\C$-linear
    triangulated category defined by
    \[
    \mathbf{DRH}(X) = \ilim[\Lambda]\mathbf{DRH}_\Lambda(X).
    \]
  \end{itemize}
\end{definition}

As a corollary of Proposition~\ref{pro:LYsmooth}, we get

\begin{theorem}
For a closed smooth Lagrangian submanifold $\Lambda \subset
X$, there is an equivalence
\[
\stack{RH}_{X,\Lambda} \simeq
\oim{p_1}\stack{LocSys}(\opb{p_1}\C_{\Omega_\Lambda^{\otimes 1/2}}),
\]
where $p_1\colon \Lambda\times \C^\times\to\Lambda$ is the projection.
\end{theorem}

\begin{remark}
When $X$ is reduced to a point, the
category of regular holonomic microdifferential modules on $X$ is
equivalent to the category of local systems on $\C^\times$.
\end{remark}

As a corollary of Theorem~\ref{th:CY}, we get

\begin{theorem}
If $X$ is compact, then $\mathbf{DRH}(X)$ is a $\C$-linear Calabi-Yau triangulated
category of dimension $\dim X+1$.
\end{theorem}

\section{Quantization algebroid}\label{se:sympl}

In this section, we first recall the construction of the
deformation-quantization algebroid of~\cite{PS04} in terms of
algebroid data.  Then, with the same data, we construct a new
$\C$-algebroid where the deformation parameter $\h$ is no longer
central. Its centralizer is related to the deformation-quantization
algebroid through a twist by the gerbe parameterizing the primitives
of the symplectic $2$-form.

\subsection{Quantization data}

Let $X$ be a complex symplectic manifold. Let $\rho =
(Y\to[\rho] X,\alpha)$ be a contactification of $X$ and $\she$ a
microdifferential algebra on $Y$.

\begin{definition}\label{def:Walg}
  A deformation parameter is an invertible section
  $\h\in\filt[-1]{\she}$ such that $\iota(\sigma_{-1}(\h)) = \alpha$,
  under the embedding \eqref{eq:OY-1}.
\end{definition}

\begin{example}
  \label{exa:stkpw}
  Let $(t;\tau)$ be
  the symplectic coordinates on $T^*\C$. Recall from
  Example~\ref{ex:PMC} the contactification of the conormal bundle $T^*M$ to a complex manifold $M$ given by
  \[
  P^*(M\times\C) \supset \{\tau\neq 0\}\to[\rho] T^*M.
  \]
  In this case the condition $\iota(\sigma_{-1}(\h)) = \alpha$ reads
  $\sigma_{-1}(\h) = \tau^{-1}$.  Denote by
  $\partial_t\in\filt[1]{\she_\C}$ the operator with total symbol
  $\tau$. It induces a deformation parameter $\h=\partial_t^{-1}$ in
  $\she_{M\times\C}$.
\end{example}

Recall that $T_\lambda\colon Y\to Y$ (for $\lambda\in\Ga$) denotes the 
$\Ga$-action on $Y$ and 
$\infGa$ denotes its infinitesimal generator.
Note that
\[
\ad(\h^{-1}) = \tfrac d{d\lambda}\Ad(e^{\lambda\h^{-1}})|_{\lambda=0}
\]
is a $\C$-linear derivation of $\she$ inducing $\infGa$ on symbols.
This derivation is integrable, and induces the isomorphism
\[
e^{\lambda\Ad(\h^{-1})} = \Ad(e^{\lambda\h^{-1}})\colon \oim{(T_{-\lambda})}\she \isoto \she.
\]
This is a quantized contact transformation over $T_{-\lambda}$.

\begin{definition}\label{def:Pw}
  The gerbe $\stkpw$ on $X$ is defined as follows.
  \begin{enumerate}
  \item Objects on $U\subset X$ are quadruples $\q =
    (\rho,\,\she,\,*,\,\h)$ of a contactification $\rho = (V \to[\rho]
    U,\alpha)$, a microdifferential algebra $\she$ on $V$, an
    anti-involution $*$ of $\she$ and a deformation parameter
    $\h\in\filt[-1]{\she}$ such that $\h^* = -\h$.

  \item If $\q' = (\rho',\she',*',\h')$ is another object,
    \begin{align*}
      \shHom[\stkpw](\q',\q) = \{(\chi,\,f)\setdef &\chi\in
      \shHom[\stkc_\omega](\rho',\rho),\
      f \in \shIso[\C\text{-}\stkalg](\oim\chi\she', \she),\\
      &f *' = * f,\ f(\h') = \h \},
    \end{align*}
    with composition given by $(\chi,f) \circ (\chi',f') =
    \bigl(\chi\chi',f(\oim\chi f')\bigr)$.
  \end{enumerate}
\end{definition}

Note that $\Ad(e^{\lambda\h^{-1}})$ commutes with $*$ for
$\lambda\in\Ga$, since $\h^* = -\h$.

\begin{remark}
  \label{re:stkpw}
Let $M$ be a complex manifold. With notations as in Example~\ref{exa:stkpw}, the operator
  $\partial_t\in\filt[1]{\she_\C}$ induces a deformation parameter $\h=\partial_t^{-1}$ in the algebra
  $\she_{\Omega_{M\times\C}^{\otimes 1/2}}$ of twisted microdifferential operators.  Hence $\stkp_{T^*M}$ has a global object given by
  \[
  (\rho,\, \she_{\Omega_{M\times\C}^{\otimes 1/2}}\bigr|_{\{\tau\neq 0\}},\,
  *,\, \partial_t^{-1}),
  \]
  with $*$ the anti-involution given by the formal adjoint.
\end{remark}

\begin{lemma}[{\cite[Lemma~5.4]{PS04}}]\label{lem:stkpw}
   For any $\q = (\rho,\,\she,\,*,\,\h)
  \in\stkpw(U)$ there is an isomorphism of sheaves of groups
  \[
  \psi \colon \Ga_U \times \{b \in
  \oim\rho\filt[0]{\she}^\times\setdef [\h,b]=0,\ b^*b=1,\
  \sigma_0(b)=1 \} \isoto \shEnd[\stkpw](\q)
  \]
  given by $\psi(\mu,b) = \bigl(T_\mu,\Ad(b
  e^{\mu\h^{-1}})\bigr)$.
\end{lemma}

One
could now try to mimic the construction of the microdifferential
algebroid $\stke_Y$ in order to get an algebroid from the algebras
$\oim\rho\she$. This fails because the automorphisms of
$(\rho,\she,*,\h)$ are not all inner, an outer automorphism being
given by $\Ad(e^{\lambda\h^{-1}})$ for $\lambda\in \C$. 

There are two
natural ways out: consider subalgebras where $\Ad(e^{\lambda\h^{-1}})$ acts as the identity, or consider bigger algebras where $\Ad(e^{\lambda\h^{-1}})$ becomes inner. 
The first solution, utilized in~\cite{PS04} to construct the deformation-quantization algebroid, is recalled in section~\ref{sse:DQ}.
The second solution is presented in section~\ref{sse:Q}, and will allow us to construct the quantization algebroid.

\subsection{Deformation-quantization algebroid}\label{sse:DQ}

Let $X$ be a complex symplectic manifold.
We can now describe
the deformation-quantization algebroid of~\cite{PS04} in terms of
algebroid data.

Let $\rho = (Y\to[\rho] X,\alpha)$ be a contactification of $X$. Let
$\she$ be a microdifferential algebra on $Y$ and $\h\in\filt[-1]\she$
a deformation parameter. To $(\rho,\she,\h)$ one associates the
deformation-quantization algebra
\[
\shw = C_\h^0 \oim\rho \she.
\]
This is the subalgebra of $\oim\rho\she$ of operators commuting with $\h$. 
Then the action of $\Ad(e^{\lambda\h^{-1}})$ is trivial
on $\shw$.

\begin{example}
As in Example~\ref{exa:stkpw}, consider the contactification of the conormal bundle $T^*M$ to a complex manifold $M$ given by
  \[
  P^*(M\times\C) \supset \{\tau\neq 0\}\to[\rho] T^*M.
  \]
Then $\h=\partial_t^{-1}$ is a deformation parameter in $\she_{M\times\C}$.
Set 
\[
\shw_M = C_{\partial_t}^0 \oim\rho \bigl( \she_{M\times\C}|_{\{\tau\neq 0\}} \bigr) .
\]
Take a local symplectic coordinate system $(x;\xi)$ on $T^*M$. Since an
element $a\in\filt[k]{\shw_M}$ commutes with $\partial_t$, its total symbol is
a formal series independent of $t$
\[
\sum_{j\leq k} \widetilde a_j(x,\xi,\tau), \quad \widetilde
a_j\in\sho_{P^*(M\times\C)}(j),
\] 
satisfying suitable growth conditions.
Setting $a_j(x,u)
= \widetilde a_{-j}(x,u,1)$ and recalling that $\h =
\partial_t^{-1}$, the total symbol of $a$ can be written as
\[
a(x,u,\h) = \sum_{j\geq -k} a_j(x,u)\h^j, \quad a_j\in\sho_{T^*M}.
\]
To make the link with usual deformation-quantization, consider two operators
$a,a'\in\filt[0]{\shw_M}$ of degree zero. Let $a(x,u)$
and $a'(x,u)$ be their respective total symbol. Then the product $aa'$ has a total symbol given
by the Leibniz star-product
\[
a(x,u) \star a'(x,u) = \sum_{\multiindex\in
\N^n} \frac{\h^{|\multiindex|}}{\multiindex !} \partial^{\multiindex}_u
a_0(x,u) \partial^{\multiindex}_x a_0'(x,u).
\]
\end{example}

Recall the
gerbe $\stkpw$ from Definition~\ref{def:Pw} and the isomorphism
$\psi$ of Lemma~\ref{lem:stkpw}.  

\begin{definition}\label{def:W}
  The deformation-quantization algebroid $\stkw_X$ is the
  $\hfield$-algebroid associated to the data
  $(\stkpw,\Phi_\stkw,\ell)$ where
  \[
  \Phi_\stkw(\q) = \shw, \quad \Phi_\stkw\bigl((\chi,\,f)\bigr) = \oim\rho f,
  \quad \ell_\q(\psi(\mu,b)) = b,
  \]
  for $\q = (\rho,\,\she,\,*,\,\h)$, $\shw = C_\h^0 \oim\rho \she$,
  $(\chi,\,f)\colon \q' \to \q$, and for $(\mu,b)$ as in Lemma~\ref{lem:stkpw}.
\end{definition}

\begin{remark}\label{re:wiii}
  Let $M$ be a complex manifold and $X=T^*M$.  With notations as in
  Remark~\ref{re:stkpw}, the algebroid $\stkw_{T^*M}$ is represented by the
  algebra $\shw_{\Omega_M^{\otimes 1/2}} = C_\h^0 \oim\rho\bigl(\she_{\Omega_{M\times\C}^{\otimes 1/2}}\bigr|_{\{\tau\neq 0\}}\bigr)$.
\end{remark}

\subsection{Quantization algebras}\label{sse:Q}

Let $\rho = (Y\to[\rho] X,\alpha)$ be a contactification of the
complex symplectic manifold $X = (X,\omega)$. Let $\she$ be a
microdifferential algebra on $Y$ and $\h\in\filt[-1]{\she}$ a
deformation parameter. Let us set
\[
\she_{[\rho]} = C_\h^\infty\oim\rho \she,
\]
where $C_\h^\infty\she = \{a \in \she\setdef \ad(\h)^N(a)=0,\ \text{locally for some }
N> 0 \}$. In local coordinates $(x,t;\xi,\tau)$, sections of
$C_\h^\infty\she$ are sections of $\she$ whose total symbol is
polynomial in $t$.

\begin{definition}\label{def:Qalg}
  The quantization algebra associated with $(\rho,\she,\h)$ is the
  $\C$-algebra
  \[
  \shquant = \DSum_{\lambda\in\Ga} \she_{[\rho]} e^{\lambda\h^{-1}}
  \]
  whose product is given by
  \[
  e^{\lambda\h^{-1}}e^{\lambda'\h^{-1}} =
  e^{(\lambda+\lambda')\h^{-1}}, \quad e^{\lambda\h^{-1}} a =
  \Ad(e^{\lambda\h^{-1}})(a) \, e^{\lambda\h^{-1}},
  \]
  for $\lambda,\lambda'\in \Ga$ and $a\in \she_{[\rho]}$.
\end{definition}

Denote by $\GRGa$ the group ring of the additive group $\C$ with
coefficients in $\C$, so that
\[
\GRGa \simeq \DSum_{\lambda\in\Ga} \C\,e^{\lambda\h^{-1}}.
\]
Then one has an algebra isomorphism
\[
C^0_\h\shquant \simeq \shw \tens[\C] \GRGa,
\]
where $\shw= \oim\rho C^0_\h\she$ is the deformation-quantization
algebra associated with $(\rho,\she,\h)$.  In particular,
$C^0_\h\shquant$ is a $\hfield \tens[\C] \GRGa$-algebra.

\subsection{Quantization algebroid}

Let $X = (X,\omega)$ be a complex symplectic manifold. Recall the
gerbe $\stkpw$ on $X$ from Definition~\ref{def:Pw} and the
isomorphism $\psi$ of Lemma~\ref{lem:stkpw}.

\begin{definition}\label{def:Q}
  The quantization algebroid on $X$ is the $\C$-algebroid
  $\stkquant_X$ associated to the data
  $(\stkpw,\Phi_\stkquant,\ell)$ where
  \[
  \Phi_\stkquant(\q) = \shquant, \quad \Phi_\stkquant(\chi,\,f) =
  \oim\rho f, \quad \ell_\q(\psi(\mu,b)) = be^{\mu\h^{-1}},
  \]
  for $\q = (\rho,\she,*,\h)$, $(\chi,\,f)\colon \q' \to \q$, and for $(\mu,b)$ as in Lemma~\ref{lem:stkpw}.
\end{definition}

Note that there is a natural action of $\C[\h]$ on $\stkquant_X$.
With the notations of \S\ref{sse:gerbes}, we set for short
\[
C^0_\h\stkquant_X = C^0_{\C[\h]}\stkquant_X.
\]

\begin{remark}
  Let $M$ be a complex manifold and $X=T^*M$.  With notations as in
  Remark~\ref{re:stkpw}, the algebroid $\stkquant_{T^*M}$ is
  represented by the algebra $\shquant_{\Omega_{M\times\C}^{\otimes 1/2}}\bigr|_{\{\tau\neq 0\}}$.
\end{remark}

Recall that $\GRGa \simeq \DSum_{\lambda\in\Ga}
\C\,e^{\lambda\h^{-1}}$. Let $\GRGa_\omega$ be the invertible
$\GRGa$-algebroid given by Definition~\ref{def:Rw} for
\[
\ell\colon \Ga \to \GRGa^\times, \quad \lambda\mapsto
e^{\lambda\h^{-1}}.
\]

The following proposition can be compared with
\cite[Remark~9.3]{PS04}.

\begin{proposition}\label{pro:WQ}
  There is an equivalence of $\hfield \tens[\C] \GRGa$-algebroids
  \[
  \stkw_X \tens[\C_X] \GRGa_\omega \simeq C^0_\h\stkquant_X.
  \]
\end{proposition}

\begin{proof}
  Consider the functor $\psi\colon C^0_\h\stkquant_X \to \stkw_X
  \tens[\C_X] \GRGa_\omega$ defined by
  \[
  (\rho,\she,*,\h) \mapsto \bigl( (\rho,\she,*,\h), \rho \bigr), \quad
  [ae^{\lambda\h^{-1}},(\chi,f)] \mapsto [a,(\chi,f)] \otimes
  [e^{\lambda\h^{-1}},\chi]
  \]
  on objects and morphisms, respectively.  Since $a\in C^0_\h\she$,
  $\psi$ is indeed compatible with composition of morphisms. To show
  that $\psi$ is an equivalence is a local problem, and thus follows
  from the isomorphism of the representative algebras
  $C^0_\h\shquant\simeq\shw\tens[\C]\GRGa$.
\end{proof}

In particular, $\stkw_X$ is equivalent to the homogeneous component of
degree zero in
\[
C^0_\h\stkquant_X \tens[\GRGa_X]\GRGa_{-\omega} \simeq \stkw_X
\tens[\C] \bigl(\DSum_{\lambda\in\Ga} \C\,e^{\lambda\h^{-1}}\bigr).
\]
Recall that $\GRGa_\omega \simeq \GRGa_X$ if $X$ admits a
contactification.

\section{Quantization modules}\label{se:mod}

Here, after establishing some algebraic properties of quantization
algebras, we show how the category $\stack{RH}_X$ of regular holonomic microdifferential modules can be embedded in the category of  quantization modules.

\subsection{A coherence criterion}

Let us state a non-commutative version of Hilbert's basis
theorem.  For a sheaf of rings $\sha$ on a topological space, consider
the sheaf of rings $\sha\langle S\rangle \simeq \sha \tens[\Z] \Z[S]$ of polynomials in a variable
$S$ which is not central but satisfies the rule
\[
Sa = \varphi(a) S + \psi(a), \quad \forall a\in\sha,
\]
where $\varphi$ is an automorphism of $\sha$ and $\psi$ is a
$\varphi$-twisted derivation, i.e.\ a linear map such that $\psi(ab) =
\psi(a)b + \varphi(a) \psi(b)$. The following result can be proved
along the same lines as \cite[Theorem~A.26]{Kas03}.

\begin{theorem}\label{thm:At}
  If $\sha$ is Noetherian, then $\sha\langle S\rangle$ is Noetherian.
\end{theorem}

\subsection{Algebraic properties of quantization algebras}

As the results in the rest of this section are of a local nature, we will consider
the geometrical situation of Example~\ref{ex:PMC}. In particular, for
$(t;\tau)$ the symplectic coordinates of $T^*\C$, we consider the
projection
\[
P^*(M\times\C) \supset Y = \{\tau\neq 0\}\to[\rho] T^*M = X.
\]
For $\h = \partial_t^{-1}$, we set
\[
\she = \she_{M\times \C}|_{\tau\neq 0}, \quad \she_{[\rho]} =
C_\h^\infty\oim\rho\she, \quad \shw = C_\h^0\oim\rho\she,
\quad \shquant = \DSum_{\lambda\in\Ga} \she_{[\rho]} e^{\lambda\h^{-1}}.
\]

\begin{theorem}
  The ring $\she_{[\rho]}$ is Noetherian.
\end{theorem}

\begin{proof}
  Note that there is an isomorphism $\shw\langle S\rangle \isoto
  \she_{[\rho]}$ given by $S\mapsto t$. Using the results of
  \cite[Appendix]{Kas03}, one proves that $\shw$ is Noetherian. Then
  $\she_{[\rho]}$ is also Noetherian by Theorem~\ref{thm:At}.
\end{proof}

\begin{theorem}\label{thm:coh}
  The sheaves of rings $\shquant$ and $C^0_\h\shquant$ are coherent.
\end{theorem}

\begin{proof}
  We shall only consider $\shquant$, as the arguments for
  $C^0_\h\shquant$ are similar.

  For a finitely generated $\Z$-submodule $\Gamma$ of $\Ga$, set
  $\shquant_\Gamma = \DSum_{\lambda\in\Gamma} \she_{[\rho]}
  e^{\lambda\h^{-1}}$. By induction on the minimal number of
  generators of $\Gamma$ one proves that $\shquant_\Gamma$ is
  Noetherian. In fact, let $\Gamma = \Gamma_0 + \Z \lambda$ and assume
  that $\shquant_{\Gamma_0}$ is Noetherian. If $\Gamma \simeq \Gamma_0
  \dsum \Z \lambda$, then $\shquant_{\Gamma_0}\langle S\rangle \isoto
  \shquant_\Gamma$ by $S\mapsto e^{\lambda\h^{-1}}$. Hence
  $\shquant_{\Gamma_0}$ is Noetherian by
  Theorem~\ref{thm:At}. Otherwise, let $N$ be the smallest integer
  such that $n\lambda\in\Gamma_0$. Then $\shquant_\Gamma \simeq
  \shquant_{\Gamma_0}\langle S\rangle/S-e^{n\lambda\h^{-1}}$ is again
  Noetherian.

  As $\shquant_\Gamma$ is Noetherian, it is in particular coherent.
  Since the morphisms $\shquant_\Gamma \to \shquant_{\Gamma'}$ are
  flat for $\Gamma \subset \Gamma'$, coherence is preserved at the
  limit $\shquant \simeq \ilim[\Gamma] \shquant_\Gamma$.
\end{proof}

For $\shm\in\stkMod(\she_{[\rho]})$, let us set for short
\[
\rho^*_\she\shm = \she\tens[\opb\rho\she_{[\rho]}] \opb\rho\shm, \quad
\Supp(\shm) = \supp(\rho^*_\she\shm) \subset Y.
\]
Let us denote by $\stkMod[\rhof,\coh](\she_{[\rho]})$ the full abelian
substack of $\stkMod[\coh](\she_{[\rho]})$ whose objects $\shm$ are such
that $\rho$ is finite on $\Supp(\shm)$.
Let us denote by $\stkMod[\rhof,\coh](\she)$ the full abelian
substack of $\stkMod[\coh](\she)$ whose objects $\shn$ are such
that $\rho$ is finite on $\supp(\shn)$.

\begin{proposition}
\label{pro:SE}
\begin{itemize}
\item[(i)]
The ring $\she$ is flat over $\opb\rho\she_{[\rho]}$.
\item[(ii)]
There is an equivalence of categories
\[
\xymatrix{\stkMod[\rhof,\coh](\she_{[\rho]})
  \ar@<.5ex>[r]^{\rho^*_\she} & \oim\rho\stkMod[\rhof,\coh](\she)
  \ar@<.5ex>[l]^{\oim\rho}, }
\]
meaning that the functors $\rho^*_\she$ and $\oim\rho$ are
quasi-inverse to each other.
\end{itemize}
\end{proposition}

Let us set for short
\begin{equation}
	\label{eq:AB}
	\sha_{k} = \opb\rho\filt[k]\she_{[\rho]},
	\qquad 
	\shb_{k} = \filt[k]\she.
\end{equation}
Note that $\sha_{-k} = \h^k\sha_0 = \sha_0\h^k$, $\shb_{-k} = \h^k\shb_0 = \shb_0\h^k$ and
\[
\sha_0/\sha_{-1} \simeq \opb\rho
\sho_X[t], \quad \shb_0/\shb_{-1} \simeq \sho_Y.  
\]
The above proposition is a non commutative analogue of the following
classical result

\begin{proposition}
  \label{pro:SEg}
\begin{itemize}
\item[(i)]
The ring $\sho_Y$ is flat over $\opb\rho\sho_X[t]$.
\item[(ii)]
There is an equivalence of categories
\[
\xymatrix{\stkMod[\rhof,\coh](\sho_X[t]) \ar@<.5ex>[r]^{\rho^*}
  & \oim\rho\stkMod[\rhof,\coh](\sho_Y)
  \ar@<.5ex>[l]^{\oim\rho}.  }
\]
\end{itemize}
\end{proposition}

\begin{proof}[Proof of Proposition~\ref{pro:SE}~(i)]
With notations \eqref{eq:AB}, it is enough to show that $\shb_0$ is flat over $\sha_0$. Thus, for a coherent
$\sha_0$-module $\shm$, we have to prove that
\begin{equation}
	\label{eq:H-1AB}
H^{-1}(\shb_0\ltens[\sha_0]\shm) =
0.
\end{equation}
One says that $u\in\shm$ is an element of $\h$-torsion if $\h^Nu=0$ for some $N\geq 0$, i.e.\ if $\sha_{-N}u=0$. Denote by $\shm^{tor}\subset\shm$ the coherent submodule
of $\h$-torsion elements. One says that $\shm$ is an $\h$-torsion module if $\shm^{tor} = \shm$ and
that $\shm$ has no $\h$-torsion if $\shm^{tor} = 0$. Considering the exact
sequence \[ 0\to \shm^{tor} \to \shm \to \shm/\shm^{tor} \to 0, \] it is
enough to prove \eqref{eq:H-1AB} in the case where $\shm$ is either an $\h$-torsion
module or has no $\h$-torsion.

  \noindent\medskip(a) Assume that $\shm$ has no $\h$-torsion.  
Then	the multiplication map
\[
\sha_{-1}\tens[\sha_0]\shm \to \shm
\]
is injective. Setting $\shm_{-1}=\sha_{-1}\shm = \h\shm$, this implies the isomorphism
\[
(\sha_0/\sha_{-1})\tens[\sha_0]\shm \simeq \shm/\shm_{-1}.
\]
By Proposition~\ref{pro:SEg}~(i), we have
\[
H^{-1}((\shb_0/\shb_{-1})\ltens[\shb_0]\shb_0\ltens[\sha_0]\shm)
\simeq H^{-1}((\shb_0/\shb_{-1})\ltens[\sha_0/\sha_{-1}](\shm/\shm_{-1})) = 0.
\]
From the exact sequence $0\to\shb_{-1}\to\shb_0\to\shb_0/\shb_{-1}\to0$ we thus obtain the exact sequence
\[
\shb_{-1}\tens[\shb_0]H^{-1}(\shb_0\ltens[\sha_0]\shm)\to H^{-1}(\shb_0\ltens[\sha_0]\shm)\to 0.
\]
By Nakayama's lemma, we get $H^{-1}(\shb_0\ltens[\sha_0]\shm) = 0$.

  \noindent\medskip(b) Let $\shm$ be an $\h$-torsion module. As $\shm$ is coherent, there locally exists $N>0$ such that $\h^N\shm = 0$. Considering the exact sequence
\[
0 \to \shm_{-1} \to \shm \to \shm/\shm_{-1}\to 0,
\]
by induction on $N$ one reduces to the case $N=1$. Then $\shm = \shm/\shm_{-1}$ has a structure of $\sha_0/\sha_{-1}$-module. Hence
\[
\shb_0\ltens[\sha_0]\shm \simeq \shb_0\ltens[\sha_0]\sha_0/\sha_{-1}\ltens[\sha_0/\sha_{-1}]\shm \simeq \shb_0/\shb_{-1}\ltens[\sha_0/\sha_{-1}]\shm,
\]
and \eqref{eq:H-1AB} follows from Proposition~\ref{pro:SEg}~(i).
\end{proof}

We shall consider an operator
$a\in\filt[0]\she_{[\rho]}$ monic in the $t$ variable, i.e.\ an operator of
the form
\begin{equation}
\label{eq:Pfin}
a = t^m + \sum_{i=0}^{m-1} b_i t^i, \qquad m\in\N_{>0},\ b_i\in\filt[0]\shw.
\end{equation}

\begin{lemma}
  \label{lem:rhorho}
  Let $a$ be of the form \eqref{eq:Pfin}. Then there are isomorphisms
\[
\rho^*_\she(\she_{[\rho]}/\she_{[\rho]}a) \simeq \she/\she a, \quad
\oim\rho(\she/\she a) \simeq \she_{[\rho]}/\she_{[\rho]}a .
\]
\end{lemma}

\begin{proof}
  The first isomorphism is clear. For the second, note that
  $\oim\rho(\she/\she a) \simeq \oim\rho\she/\oim\rho\she a$ since
  $\rho$ is finite on $\supp(\she/\she a)$. Note also that, by
  division, any $c \in \oim\rho\she$ can be written as $c = d a + b$
  with $d\in\oim\rho\she$ and $b\in\she_{[\rho]}$. Then the
  isomorphism $\oim\rho\she/\oim\rho\she a \isoto
  \she_{[\rho]}/\she_{[\rho]}a$ is given by $c \mapsto b$.
\end{proof}

\begin{proof}[Proof of Proposition~\ref{pro:SE}~(ii)]
  (a) Let $\shn_0$ be a coherent $\filt[0]\she$-module such that $\rho$ is finite on $\supp\shn_0$. We will show that $\shn_0$ is $\filt[0]\shw$-coherent. As this is a local problem on $Y$, we can assume that $(x_0,t;\xi_0,1)\in\supp\shn_0$ only for $t=0$.
	Thus $\supp\shn_0\subset\{t^p+\varphi(x,t,\xi/\tau)=0\}$ with $\varphi\in\sho_X[t]$ vanishing for $t=0$ and of degree less than $p$  in the $t$ variable. Choose a
	system $u_1,\dots,u_N$ of generators for $\shn_0$. By division, for
	each $i$ there exists $a_i$ of the form \eqref{eq:Pfin} such that
	$a_iu_i=0$. One thus gets an exact sequence
	\[
	0 \to \shn_0' \to \DSum_{i=1}^N \filt[0]\she/\filt[0]\she a_i \to \shn_0 \to 0.
	\]
	As $\filt[0]\she/\filt[0]\she a_i$ is $\filt[0]\shw$-coherent, $\shn_0$ is a finitely generated $\filt[0]\shw$-module. Since also $\shn_0'$ is finitely generated over $\filt[0]\shw$, it follows that $\shn_0$ is $\filt[0]\shw$-coherent.
	
	In particular, this shows that any $\shn\in\oim\rho\stkMod[\rhof,\coh](\she)$ is a coherent
	$\she_{[\rho]}$-module.
	
  (b) Let $\shn\in\oim\rho\stkMod[\rhof,\coh](\she)$ and choose a
  system $u_1,\dots,u_N\in\shn$ of generators. By (a), 
$\oim\rho\filt[0]\she u_i$ is $\filt[0]\shw$-coherent. Hence, $\{t^j\filt[0]\shw u_i\}_{j>0}$ is stationary in $\oim\rho\filt[0]\she u_i$, so that there exist $m_i> 0$ and $b_{ij}\in\filt[0]\shw$ such that $t^{m_i} u_i = \sum_{j<m_i} b_{ij} t^j u_i$. In other words,
for
  each $i$ there exists $a_i = t^{m_i} - \sum_j b_{ij} t^j$ of the form \eqref{eq:Pfin} such that
  $a_iu_i=0$. One thus gets an exact sequence
\[
0 \to \shn' \to \DSum_{i=1}^N \she/\she a_i \to \shn \to 0.
\]
Applying the same argument to $\shn'$ one gets a presentation
\[
\DSum_{i=1}^{N'} \she/\she a'_i \to \DSum_{i=1}^N \she/\she a_i \to
\shn \to 0.
\]
Since $\oim\rho = \eim\rho$ is exact on this sequence, by
Lemma~\ref{lem:rhorho} the module $\oim\rho\shn$ has the presentation
\[
\DSum_{i=1}^{N'} \she_{[\rho]}/\she_{[\rho]} a'_i \to \DSum_{i=1}^N
\she_{[\rho]}/\she_{[\rho]} a_i \to \oim\rho\shn \to 0.
\]
Applying the exact functor $\rho^*_\she$ and
using again Lemma~\ref{lem:rhorho}, we get that
$\rho^*_\she\oim\rho\shn \isoto \shn$.

\smallskip\noindent (c) For
$\shm\in\stkMod[\rhof,\coh](\she_{[\rho]})$, let us show that the map
$\shm\to\oim\rho\rho^*_\she\shm$ is injective. Let $\shm_0$ be a
lattice of $\shm$, that is a coherent
sub-$\filt[0]\she_{[\rho]}$-module such that $\she_{[\rho]}\shm_0 =
\shm$. Since $\rho^*_{\filt[0]\she}\shm_0$ is a lattice for
$\rho^*_\she\shm$, it is enough to prove the injectivity of the map
$\shm_0\to\oim\rho\rho^*_{\filt[0]\she}\shm_0$.  Assume that
$u\in\shm_0$ is sent to $0$.  By Proposition~\ref{pro:SEg} there are
isomorphisms
\[
\shm_0/\filt[-1]\she\shm_0 \isoto
\oim\rho\rho^*(\shm_0/\filt[-1]\she\shm_0) \simeq
\oim\rho\rho^*_{\filt[0]\she}\shm_0/\filt[-1]\she\oim\rho\rho^*_{\filt[0]\she}\shm_0.
\]
It follows that $u\in \filt[-1]\she\shm_0$. By induction we then get
$u\in\bigcap_{k>0}\filt[-k]\she\shm_0$, so that $u = 0$.

\smallskip\noindent (d) We finally have to prove the isomorphism
$\shm\isoto\oim\rho\rho^*_\she\shm$.  Let $u_1,\dots,u_N$ be a system
of generators of $\shm$. By the same arguments as in (b), for each $i$ there
exists $a_i$ of the form \eqref{eq:Pfin} such that $a_iu_i=0$ in
$\rho^*_\she\shm$. By (c) this implies $a_iu_i=0$ in $\shm$. As in
(b) we thus get a resolution
\[
\DSum_{i=1}^{N'} \she_{[\rho]}/\she_{[\rho]} a'_i \to \DSum_{i=1}^N
\she_{[\rho]}/\she_{[\rho]} a_i \to \shm \to 0,
\]
giving the isomorphism $\shm\isoto\oim\rho\rho^*_\she\shm$ by
Lemma~\ref{lem:rhorho}.
\end{proof}

For $S\subset Y$, let us denote by $\stkMod[S,\coh](\she_{[\rho]})$ the full
abelian substack of $\stkMod[\coh](\she_{[\rho]})$ whose objects $\shm$ are such
that $\Supp(\shm) \subset S$. 
For $T \subset X$, let us denote by $\stkMod[T,\coh](\shquant)$ the full
abelian substack of $\stkMod[\coh](\shquant)$ whose objects $\shm$ are such
that $\supp(\shm) \subset T$.

We set for short
\[
\shquant\shm = \shquant\tens[\she_{[\rho]}] \shm.
\]

\begin{proposition}\label{pr:SQ}
	\begin{itemize}
	\item[(i)]
	The ring $\shquant$ is faithfully flat over $\she_{[\rho]}$.
	\item[(ii)]
  Let $S\subset Y$ be an analytic subset such that $\rho|_S$ is proper
  and injective. Then the functor
\[
\shquant(\cdot)\colon \stkMod[S,\coh](\she_{[\rho]}) \to
\stkMod[\rho(S),\coh](\shquant_X)
\]
is fully faithful.
\end{itemize}
\end{proposition}

\begin{proof}
	(i) is straightforward.

\smallskip\noindent (ii) 
For a coherent $\she_{[\rho]}$-module $\shm$, there is an isomorphism of $\she_{[\rho]}$-modules
\[
\shquant \shm \simeq \DSum_{\lambda\in\Ga} e^{\lambda\h^{-1}} \shm.
\]
Here, the $\she_{[\rho]}$-module structure of $e^{\lambda\h^{-1}}
\shm$ is given by
\[ 
a (e^{\lambda\h^{-1}}\cdot b) =
e^{\lambda\h^{-1}}\cdot\Ad(e^{-\lambda\h^{-1}})(a)b,
\] 
for $a\in\she_{[\rho]}$ and $b\in\shm$. 
Note that
$\Supp(e^{\lambda\h^{-1}}\shm) = T_\lambda \Supp(\shm)$.

For $\shm,\shm' \in \stkMod[S,\coh](\she_{[\rho]})$, one has
\begin{align*}
  \Hom[\shquant](\shquant \shm',\shquant \shm) & \simeq
  \Hom[\she_{[\rho]}](\shm',\DSum_{\lambda\in\Ga} e^{\lambda\h^{-1}} \shm) \\
  & \simeq \DSum_{\lambda\in\Ga} \Hom[\she_{[\rho]}](\shm', e^{\lambda\h^{-1}} \shm) \\
  & \simeq \DSum_{\lambda\in\Ga} \Hom[\she](\rho^*_\she\shm',
  \rho^*_\she
  (e^{\lambda\h^{-1}}\shm)) \\
  & \simeq \Hom[\she](\rho^*_\she\shm', \rho^*_\she\shm) \\
  & \simeq \Hom[\she_{[\rho]}](\shm',\shm),
\end{align*}
where the second last isomorphism is due to the fact that
$\Supp(\shm') \cap \Supp(e^{\lambda\h^{-1}} \shm) = \emptyset$ for
$\lambda\neq 0$.
\end{proof}

\subsection{Induced modules}\label{sse:ind}

Assume that the symplectic manifold $X$ admits a
contactification $\rho = (Y \to[\rho] X,\alpha)$. In this section we
show how the constructions from the previous section can be
globalized.

\begin{definition}\label{def:Pr}
  For a contactification $\rho$ of $X$, the gerbe $\stkpr$ on $X$ is
  defined as follows.
  \begin{enumerate}
  \item Objects on $U\subset X$ are triples $\pq = (\she,\,*,\,\h)$ of
    a microdifferential algebra $\she$ on $\rho^{-1}(U)$, an
    anti-involution $*$ of $\she$ and a deformation parameter $\h$
    such that $\h^* = -\h$.

  \item If $\pq' = (\she',*',\h')$ is another object,
    \[
    \shHom[\stkpr](\pq',\pq) = \{f\in
    \shIso[\shr\text{-}\stkalg](\she', \she) \setdef f *' = * f,\
    f(\h') = \h \}.
    \]
  \end{enumerate}
\end{definition}

As a corollary of Lemma~\ref{lem:stkpe}, one has

\begin{lemma}
  For any $\pq = (\she,\,*,\,\h) \in \stkpr$ there is an isomorphism
  of sheaves of groups
  \[
  \psi_\rho\colon \{b \in \she^\times \setdef [\h,b]=0,\ b^*b=1,\
  \sigma_0(b)=1 \} \isoto \shEnd[\stkpr](\pq)
  \]
  given by $\psi_\rho(b) = \Ad(b)$.
\end{lemma}

\begin{definition}\label{def:Er}
  For a contactification $\rho$ of $X$, the stack $\stke_{[\rho]}$ is
  the $\C$-algebroid associated to the data
  $(\stkpr,\Phi_{\stke_{[\rho]}},\ell)$ where
  \[
  \Phi_{\stke_{[\rho]}}(\pq) = \she_{[\rho]}, \quad
  \Phi_{\stke_{[\rho]}}(f) = \oim\rho f, \quad \ell_{\pq}(g) = b,
  \]
  for $\pq = (\she,\,*,\,\h)$, $f\colon \pq' \to \pq$ and $g =
  \psi_\rho(b)$.
\end{definition}

Note that Proposition~\ref{pro:WQ} implies $\stkw_X\simeq
C^0_\h\stke_{[\rho]}$.

As in the local case, for $\shm\in\stkMod(\stke_{[\rho]})$ we set for
short
\[
\Supp(\shm) = \supp(\rho^*_\stke\shm) \subset Y.
\]
Consider the faithful $\C$-linear functors
\begin{align*}
  \opb\rho\stke_{[\rho]} &\to \stke_Y, & (\she,*,\h) &\mapsto
  (\she,*),
  &&\text{on objects}, \\
  && (a,f) & \mapsto (a,f), &&\text{on morphisms}, \\
  \stke_{[\rho]} &\to \stkquant_X, & (\she,*,\h) &\mapsto
  (\rho,\she,*,\h), &&\text{on objects}, \\
  &&(a,f) &\mapsto (ae^{0\h^{-1}},\id_\rho,f), &&\text{on morphisms}.
\end{align*}
For $S\subset Y$ they induce the functors
\begin{align*}
  \rho^*_\stke&\colon\stkMod[\rhof,\coh](\stke_{[\rho]}) \to
  \oim\rho\stkMod[\rhof,\coh](\stke_Y),\\
  \stkquant(\cdot)&\colon\stkMod[S,\coh](\stke_{[\rho]}) \to
  \stkMod[\rho(S),\coh](\stkquant_X).
\end{align*}

By Propositions~\ref{pro:SE} and \ref{pr:SQ} we have

\begin{proposition}\label{pro:SEglobal}
  (i) The functor $\rho^*_\stke$ is an equivalence.

  \smallskip\noindent (ii) Let $S\subset Y$ be an analytic subset such
  that $\rho|_S$ is proper and injective. Then $\stkquant(\cdot)$ is
  fully faithful.
\end{proposition}

We can thus embed regular holonomic microdifferential modules in the
stack of coherent $\stkquant_X$-modules.

\begin{corollary}
  There is a fully faithful embedding
  \[
  \stack{RH}_X \subset \stkMod[\coh](\stkquant_X).
  \]
\end{corollary}

\begin{remark}
  We do not know if the above result extends to give an embedding
  $\mathbf{DRH}(X)\subset\BDC_{\coh}(\stkquant_X)$ at the level of
  derived categories.
\end{remark}

\appendix

\section{Remarks on deformation-quantization}\label{se:dq}

We give in this appendix an alternative description of the deformation
quantization algebroid using triples $(\shw,*,v)$ of a
deformation-quantization algebra $\shw$ endowed with an
anti-involution $*$ and an order preserving $\C$-linear derivation
$v$. We also compare regular holonomic deformation-quantization
modules with regular holonomic quantization modules.

\subsection{Deformation-quantization and derivations}

Let $X = (X,\omega)$ be a complex contact manifold and $\shw$ a
deformation quantization algebra on $X$.

\begin{lemma}\label{lem:ChDer}
  Let $w$ be an order preserving $\hfield$-linear derivation of
  $\shw$. Then $w$ is locally of the form $\ad(\h^{-1}d)$ for some
  $d\in\filt[0]{\shw}$.
\end{lemma}

\begin{proof}
  Let $(x;u)$ be a local system of quantized symplectic coordinates
  (see~\cite[\S2.2.3]{KR08}). For $i=1,\dots, n$, set $e_i = \h
  w(x_i)\in\filt[-1]{\shw}$. From $w([x_i,x_j]) = 0$ we get $[e_i,x_j]
  = [e_j,x_i]$ for any $i,j=1,\dots, n$. Hence there locally exists
  $e\in\filt[0]{\shw}$ with $e_i = [x_i,e]$. Replacing $w$ by
  $w-\ad(\h^{-1}e)$ we may assume $w(x_i) = 0$.

  Set $d_i = \h w(u_i)\in\filt[-1]{\shw}$. From $w([x_i,u_j]) = 0$ we
  get $[x_i,d_j] = 0$, so that $d_i = d_i(x)$ does not depend on
  $u$. From $w([u_i,u_j]) = 0$ we get $[d_i,u_j] = [d_j,u_i]$. Hence
  there locally exists $d = d(x)\in\filt[0]{\shw}$ with $d_i =
  [u_i,d]$. Replacing $w$ by $w-\ad(\h^{-1}d)$ we have $w(x_i) =
  w(u_j) = 0$, and hence $w=0$.
\end{proof}

\begin{definition}\label{def:Pai}
  Let $\stkpa$ be the stack on $X$ associated with the separated
  prestack $\stkpao$ defined as follows.
  \begin{enumerate}
  \item Objects on $U\subset X$ are triples $\dq = (\shw,\,*,\,v)$ of
    a deformation quantization algebra $\shw$ on $U$, an
    anti-involution $*$ and an order preserving $\C$-linear derivation
    $v$ of $\shw$ such that $v(\h) = \h$ and $v * = * v$.

  \item If $\dq' = (\shw',\,*',\,v')$ is another object,
    \begin{multline*}
      \shHom[\stkpao](\dq',\dq) = \{(g,d)\setdef
      g\in\shIso[\shr\text{-}\stkalg](\shw', \shw),\ d\in\filt[0]{\shw},\\
      g *' = * g,\ d = d^*,\ v-gv'g^{-1} = \ad(\h^{-1}d) \},
    \end{multline*}
    with composition given by $(g,d) \circ (g',d') = (gg',d+g(d'))$.
  \end{enumerate}
\end{definition}

Using Lemma~\ref{lem:ChDer} one gets

\begin{lemma}
  The stack $\stkpa$ is a gerbe.
\end{lemma}

\begin{remark}
  Let $M$ be a complex manifold and $X=T^*M$. With notations as in
  Remark~\ref{re:stkpw}, where $\h = \partial_t^{-1}$, a global
  object of $\stkpa$ is given by
  $(\shw_{\Omega_M^{\otimes 1/2}},*,\ad(t\partial_t))$.
\end{remark}

\begin{lemma}\label{lem:stkpwi}
  For any $\dq = (\shw,\,*,\,v)\in\stkpa(U)$ there is a group
  isomorphism
  \[
  \psi'_\omega \colon \Ga_U \times \{ b \in \filt[0]{\shw}^\times
  \setdef b^*b=1,\ \sigma_0(b)=1 \} \isoto \shEnd[\stkpa](\dq)
  \]
  given by $\psi'_\omega(\mu,b) = (\Ad(b),\mu + \h v(b)b^{-1})$.
\end{lemma}

\begin{proof}
  (i) Let us prove injectivity. Assume that $\Ad(b)=\id$ and $\mu+ \h
  v(b)b^{-1} =0$. Then $b\in\hfieldo$, $\mu=0$ and $v(b)=0$. As $v(b)
  = \h\frac{\partial}{\partial\h}b$, we get $b\in\C$. Since
  $\sigma_0(b) = 1$, this finally gives $b=1$.

  \medskip\noindent (ii) Let us prove surjectivity. Take $(g,d)\in
  \shEnd[\stkpa](\dq)$. Since any $\hfield$-algebra automorphisms of
  $\shw$ is inner, we can locally write $g = \Ad(b)$ for some
  $b\in\filt[0]{\shw}^\times$. As $g$ commutes with the
  anti-involutions, we have $\Ad(b)(a^*) = (\Ad(b)(a))^* =
  \Ad(b^{*-1})(a^*)$ for any $a\in\shw$. This implies $\Ad(b^* b) =
  \id$, so that $b^*b\in\hfieldo$. Take $k\in\hfieldo$ with $k^*k =
  b^*b$. Up to replacing $b$ with $bk^{-1}$ we may thus assume that
  $b^*b=1$. This implies $\sigma_0(b) = \pm1$ and we may further
  assume that $\sigma(b) = 1$. Replacing $(g,d)$ by
  $(g,d)\cdot\psi'_\omega(b^{-1},0)$ we may thus assume $g = \id$.

  Since $\ad(\h^{-1}d) =0$, we have $d\in\hfieldo$. As $d^*=d$ and
  $\h^* = -\h$, the coefficients of the odd powers of $\h$ in $d$
  vanish, and we may write $d = \mu + \h^2 d'$ for $\mu\in\C$ and
  $d'\in\hfieldo$. Take $d''\in\hfieldo$ such that
  $\h\frac{\partial}{\partial\h}d'' = d'$, and set $b=\exp(\h d'')$.
  Since $v(b)b^{-1} = \h d'$, we have $d = \mu + \h v(b)b^{-1}$. Hence
  $\psi'_\omega(\mu,b) = (\id, d)$.
\end{proof}

\begin{definition}\label{def:Wprime}
  The algebroid $\stkw'_X$ is the $\hfield$-algebroid associated to
  the data $(\stkpa,\Phi'_\stkw,\ell)$ where
  \[
  \Phi'_\stkw(\dq) = \shw, \quad \Phi'_\stkw(g,d) = g, \quad
  \ell_{\dq}(h,e) = b,
  \]
  for $\dq = (\shw,\,*,\,v)$, $(g,d)\colon \dq' \to \dq$ and $(h,e) =
  \psi'_\omega(\mu,b)$.
\end{definition}

\begin{proposition}
  There is a $\hfield$-linear equivalence
  \[
  \stkw'_X \simeq \stkw_X.
  \]
\end{proposition}

This follows from the following proposition.

\begin{proposition}
  There is an equivalence of gerbes
  \[
  \stkpa \simeq \stkpw.
  \]
\end{proposition}

\begin{proof}
  Let us consider the gerbe $\stkpw''$ whose objects on
  $U\subset X$ are quintuples $\q = (\rho,\,\she,\,*,\,\h,\hh)$ such
  that $\pi(\q) = (\rho,\,\she,\,*,\,\h)$ is an object of
  $\stkpw$ and $\hh\in\filt[0]{\she}$ is an operator with
  $[\h^{-1},\hh] = 1$. (The local model in a Darboux chart is obtained
  by Example~\ref{exa:stkpw} with $\h^{-1}=\partial_t$ and $\hh = t$.)
  We set
  \[
  \shHom[\stkpw''](\q',\q) =
  \shHom[\stkpw](\pi(\q'),\pi(\q)).
  \]
  There is a natural equivalence
  \[
  \stkpw'' \isoto \stkpw, \quad \q \mapsto \pi(\q).
  \]
  Consider the functor $\psi\colon \stkpw'' \to \stkpa$ given by
  \begin{align*}
    \q & \mapsto (C^0_\h\oim\rho\she,*,\ad(\hh \h^{-1})), &&
    \text{for }\q = (\rho,\,\she,\,*,\,\h,\hh), \\
    (\chi,f) & \mapsto \bigl(\oim\rho f,\hh-f(\hh')\bigr), &&\text{for
    }(\chi,f)\colon\q'\to\q.
  \end{align*}
  This is well defined since
  \[
  \ad(\hh \h^{-1}) - f \ad(\hh' \h^{\prime -1}) f^{-1} =
  \ad((\hh-f(\hh'))\h^{-1}).
  \]
  It follows from Lemmas~\ref{lem:stkpwi} and \ref{lem:stkpw} that
  $\psi$ is fully faithful. As $\stkpw''$ and $\stkpa$ are
  gerbes, $\psi$ is an equivalence.
\end{proof}

Recall that if $\dq = (\shw,\,*,\,v)$ is an object of $\stkpa$ on an
open subset $U\subset X$, then $\stkw_X|_U$ is represented by
$\shw$. As shown in~\cite{Pol08}, the filtration and the
anti-involution of $\shw$ extend to $\stkw_X$. As we will now explain,
also the derivation of $\shw$ extends to $\stkw_X$.

Let $\varepsilon$ be a formal variable with $\varepsilon^2 =
0$. Consider the natural morphisms
\[
\shw \to[i] \shw[\varepsilon] \to[\pi] \shw.
\]
Let us extend the anti-involution $*$ to $\shw[\varepsilon]$ by
setting $\varepsilon^* = -\varepsilon$.

\begin{lemma}
  Let $\varphi\colon\shw\to\shw[\varepsilon]$ be an order preserving
  $\C$-algebra morphism such that $\pi\varphi = \id_\shw$,
  $\varphi(\h) = \h +\varepsilon\h^2$ and $\varphi * = *
  \varphi$. Then $\varphi = i + \varepsilon\h v$ for an order
  preserving $\C$-linear derivation $v$ of $\shw$ such that $v* = *v$.
\end{lemma}

\begin{remark}
  There is an isomorphism of $\shw\tens[\C]\shw^\op$-modules
  \[
  (\shw[\varepsilon])_\varphi \simeq C_\h^1\oim\rho\she
  \]
  such that the multiplication by $\varepsilon$ corresponds to
  $\ad(\h^{-1})$.  In local coordinates where $\h^{-1} = \partial_t$
  and $v=\ad(t\partial_t)$, this isomorphism is given by
  $a+\varepsilon b \mapsto at+b$.
\end{remark}

The above lemma motivates the following definition.

\begin{definition}
  A derivation of a $\C$-linear stack $\stka$ is the data of a pair
  $\varphi=(\stkc,\varphi)$ where $\stkc$ is an invertible
  $\C[\varepsilon]$-algebroid such that $\stkc/\varepsilon$ is
  represented by $\C_X$ and $\varphi\colon \stka \to \stka \tens[\C]
  \stkc$ is a $\C$-linear functor such that $\pi\varphi \simeq
  \id_\stka$. Here $\pi\colon \stka\tens[\C]\stkc \to \stka$ is the
  functor induced by $\stkc\to\stkc/\varepsilon$.
\end{definition}

Consider the following algebroid.

\begin{definition}
  The algebroid $\stkw_X^\varepsilon$ is the
  $\hfield[\varepsilon]$-algebroid associated to the data
  $(\stkpa,\Phi_\stkw^\varepsilon,\ell)$ where
  \[
  \Phi_\stkw^\varepsilon(\dq) = \shw[\varepsilon], \quad
  \Phi_\stkw^\varepsilon(g,d) = (1+\varepsilon\ad(d))g, \quad
  \ell_{\dq}(h,e) = (1+\varepsilon\mu) b,
  \]
  for $\dq = (\shw,\,*,\,v)$, $(g,d)\colon \dq' \to \dq$ and $(h,e) =
  \psi'_\omega(\mu,b)$.
\end{definition}

There is a natural morphism
\[
\varphi\colon \stkw_X \to \stkw_X^\varepsilon
\]
satisfying $\varphi(\h) = \h +\varepsilon\h^2$ and $\varphi * = *
\varphi$.  Similarly to Proposition~\ref{pro:WQ}, one proves that
there is an equivalence of $\hfield[\varepsilon]$-algebroids
\[
\stkw_X^\varepsilon \simeq \stkw_X \tens[\C] \C[\varepsilon]_\omega,
\]
where $\C[\varepsilon]_\omega$ is the invertible
$\C[\varepsilon]$-algebroid given by Definition~\ref{def:Rw} for
\[
\ell\colon \Ga \to \C[\varepsilon]^\times, \quad \lambda\mapsto
(1+\varepsilon\lambda).
\]
Thus $\stkw_X$ is endowed with the derivation
$\varphi=(\C[\varepsilon]_\omega,\varphi)$.

Summarizing, $\stkw_X$ is a filtered $\hfield$-stack endowed with an
anti-involution $*$ and with a $\C$-linear derivation $\varphi$ such
that $\filt[0]\stkw_X/\filt[-1]\stkw_X$ is represented by $\sho_X$,
$\varphi(\h) = \h$ and $\varphi *=*\varphi$. One can prove along the
lines of~\cite{Pol08} that $\stkw_X$ is unique among the stacks which
satisfy these properties and which are locally represented by
deformation quantization algebras.

\subsection{Comparison of regular holonomic modules}

We shall compare here regular holonomic quantization-modules with regular holonomic deformation-quantization modules.
Let us start by recalling the definition of regular holonomic quantization-modules from~\cite{KS08}.

Let $X$ be a complex symplectic manifold and $\Lambda$ a closed Lagrangian subvariety of $X$.
Let $\shw$ be a deformation-quantization algebra on $X$.

\begin{definition}
\begin{itemize}
\item[(i)]
One says that a coherent $\filt[0]\shw$-module $\shm_0$ is regular holonomic along
$\Lambda$ if $\supp(\shm_0)\subset\Lambda$ and
$\shm_0/\h\shm_0$ is a coherent $\sho_\Lambda$-module.
\item[(ii)]
One says that a coherent $\shw$-module $\shm$ is regular holonomic along $\Lambda$ if 
$\supp(\shm)\subset\Lambda$ and
there exists locally a coherent $\filt[0]\shw$-submodule $\shm_0$
of $\shm$ such that $\shm_0$  generates $\shm$ over $\shw$ and $\shm_0$ is 
regular holonomic along $\Lambda$.
\end{itemize}
\end{definition}

Recall that $\stkw_X$ denotes the deformation-quantization algebroid.
As the above definition is local, there is a natural notion of regular holonomic $\stkw_X$-module along $\Lambda$.
Let us denote by $\stkMod[\Lambda,\reghol](\stkw_X)$
the full substack of $\stkMod[\coh](\stkw_X)$ whose objects are regular holonomic along $\Lambda$.

Up to shrinking $X$, we may assume that there exist a
contactification $\rho\colon Y\to X$ and a Lagrangian subvariety
$\Gamma$ of $Y$ such that $\rho$ induces an isomorphism $\Gamma \to
\Lambda$. By definition, regular holonomic $\stkquant_X$-modules along
$\Lambda$ are equivalent to regular holonomic $\stke_Y$-modules along
$\Gamma$. In order to compare quantization and
deformation-quantization modules, let us thus consider the forgetful
functor
\[
\mathrm{for}\colon \oim\rho\stkMod[\Gamma,\reghol](\stke_Y) \to
\stkMod[\Lambda,\reghol](\stkw_X)
\]
induced by the equivalence $\stkw_X \simeq C^0_\h\stke_{[\rho]}$ and the functor $\opb\rho\stke_{[\rho]} \to \stke_Y$ from \S\ref{sse:ind}.

\begin{proposition}
  \begin{itemize}
  \item[(i)] The functor $\mathrm{for}$ is faithful but not locally
    full in general.
  \item[(ii)] If $\Lambda$ is a smooth submanifold, the functor
    $\mathrm{for}$ is locally essentially surjective but not
    essentially surjective in general.
  \item[(iii)] The functor $\mathrm{for}$ is not locally essentially
    surjective in general.
  \end{itemize}
\end{proposition}

\begin{proof}
  (i) holds more generally for the forgetful functor
  $\oim\rho\stkMod(\stke_Y) \to \stkMod(\stkw_X)$.

  \medskip\noindent (ii) Let $\Lambda$ be a smooth
  submanifold. Consider the commutative diagram
  \[
  \xymatrix{ \oim\rho\stkMod[\Gamma,\reghol](\stke_Y) \ar@{<->}[d]^\sim
    \ar[r]^{\mathrm{for}} &
    \stkMod[\Lambda,\reghol](\stkw_X) \ar@{<->}[d]^\sim \\
    \oim\rho\oim{p_1}\stack{LocSys}(\opb{p_1}\C_{\Omega_\Gamma^{\otimes 1/2}})
    \ar[r] & \stack{LocSys}(\hfield_{\Omega_\Lambda^{\otimes 1/2}}), }
  \]
  where $p_1\colon\Gamma \times \C^\times\to\Gamma$ is the
  projection.  The vertical equivalences are due to
  Proposition~\ref{pro:LYsmooth} and \cite[Corollary~9.2]{DS07},
  respectively. The bottom arrow is given by $L \mapsto \hfield
  \tens[\C] L|_{s = 1}$, where $s$ is the coordinate of $\C^\times$.

  This shows that the forgetful functor is locally essentially
  surjective.  To prove that it is not surjective in general, take
  $X=\C^\times$ and $\Lambda$ the zero section of
  $T^*(\C^\times)$. Then the local system with monodromy $1+\h$ around
  the origin is not in the essential image of the forgetful functor.

  \medskip\noindent (iii) follows from Proposition~\ref{pro:conex}
  below.
\end{proof}

Before stating Proposition~\ref{pro:conex} let us introduce some
notations.

Let $M=\C$. Denote by $(x,t;\xi,\tau)$ the symplectic coordinates of
$P^*(M\times\C)$ and by $(x;u)$ those of $T^*M$. Let $\shw = \shw_M$,
and recall that $\h = \partial_t^{-1}$. We will identify elements
$a\in\shw$ with their total symbol $a(x,u,\tau)$, and write for
example $a_x$ for the operator with total symbol
$\frac{\partial}{\partial x} a(x,u,\tau)$.

Denote by $\sho_M^\hbar = \shw/\shw \partial_x$ the canonical regular
holonomic module along the zero section
\[
\Lambda_1 = \{(x,u)\setdef u=0\}.
\]
The quotient map $\shw \to \sho_M^\hbar$, $b \mapsto [b]$ induces an
isomorphism of vector spaces $\sho_M^\hbar \xleftarrow{\sim}
C^0_{x}\shw$ with the subring of operators whose total symbol does not
depend on $\partial_x$.

For $m \in \Z_{>0}$, consider the Lagrangian subvariety $\Lambda =
\Lambda_1 \cup \Lambda_2$, with
\[
\Lambda_2 = \{(x,u)\setdef u=x^m\}.
\]
For $a \in C^0_{x}\shw$, let $\shm_a$ be the regular holonomic module
along $\Lambda$ with generators $v_1$, $v_2$ and relations
\[
\partial_x v_1 = 0, \quad (\partial_x - x^m\partial_t) v_2 = a v_1.
\]
Note that
\[
\shm_a \simeq C^0_{x}\shw\, v_1 \dsum C^0_{x}\shw\, v_2.
\]

Let $a' \in C^0_{x}\shw$ be another operator. If $[a-a']\in
(\partial_x - x^m\partial_t)\sho_M^\hbar$, then $\shm_a \isoto
\shm_{a'}$. In fact, if $e\in C^0_{x}\shw$ satisfies $a-a' = e_x -
x^m e \partial_t$, an isomorphism $\shm_a \isoto
\shm_{a'}$ is given by 
$v_1 \mapsto v_1'$, $v_2 \mapsto v'_2 + e v'_1$. 
Since $\sho_M^\hbar/(\partial_x -
x^m\partial_t)\sho_M^\hbar \simeq \DSum_{i=0}^{m-1}\hfield x^i$, we
may thus assume that
\[
a = a_0 + a_1 x + \cdots + a_{m-1} x^{m-1}\quad\text{with $a_i\in\hfield$.}
\]
The following counterexample was developed by the second author (M.K.)
while working with Pierre Schapira at~\cite{KS08}.

\begin{proposition}\label{pro:conex}
  If $\shm_a \simeq \mathrm{for}(\shn)$ for some $\she_Y$-module
  $\shn$, then $a$ is homogeneous, i.e.~$a = a_{i_0} x^{i_0}$ for some
  $i_0\in\{0,\dots, m-1\}$.
\end{proposition}

\begin{proof}
  The existence of such an $\shn$ is equivalent to the existence of an
  endomorphism $t$ of $\shm_a$ such that $[t,x] = [t,\partial_x] = 0$
  and $[t,\partial_t] = -1$.

  \smallskip\noindent (i) Let $t v_1 = b v_1 + c v_2$ for $b,c\in
  C^0_{x}\shw$. Then
  \begin{align*}
    0 &= t \partial_x v_1 = \partial_x t v_1 \\
    &= \partial_x (b v_1 + c v_2) \\
    &= b_x v_1 + c_x v_2 + c(x^m \partial_t v_2 + a v_1).
  \end{align*}
  Hence
  \[
  b_x + ac = 0, \quad x^m c\partial_t + c_x = 0.
  \]
  It follows from the second equation that $c=0$. Thus the first
  equation implies that $b\in\hfield$. Up
  to replacing $t$ by $t-b$, we may assume that $t v_1 = 0$.

  \smallskip\noindent (ii) Let $t v_2 = b v_1 + c v_2$ for $b,c\in
  C^0_{x}\shw$. Then
  \begin{align*}
    0 &= t \bigl( (\partial_x - x^m\partial_t) v_2 - a v_1 \bigr) \\
    &= (\partial_x - x^m\partial_t) t v_2 + x^m v_2 - [t,a] v_1 \\
    &= (\partial_x - x^m\partial_t)(b v_1 + c v_2) + x^m v_2 - [t,a] v_1 \\
    &= b_x v_1 + c_x v_2 + c(x^m\partial_t v_2 + a v_1) - x^m
    b\partial_t v_1 - x^m c \partial_t v_2 + x^m v_2 - [t,a]v_1.
  \end{align*}
  Hence
  \begin{equation}
\label{eq:abc}
  ac + b_x - x^m b \partial_t - [t,a] = 0, \quad c_x + x^m = 0.
  \end{equation}
  The second equation gives $c = -\frac{x^{m+1}}{m+1} + d$ for $d\in
  \hfield$.  Then, the first equation in \eqref{eq:abc}
  can be rewritten
  \[
  (\ad(\partial_x) - x^m\partial_t)(xa + (m+1)b\partial_t) 
- (xa_x - ea + (m+1)\partial_t[t,a]) =  0,
  \]
  for $e = (m+1)d\partial_t -1\in\hfield$. Hence 
$xa+(m+1)b\partial_t = x a_x -ea + (m+1)\partial_t[t,a] = 0$.
Since $a =
  \sum_{i=0}^{m-1} a_i x^i$, it implies that
$\sum_{i=0}^{m-1} ((e-i) a_i - (m+1)\partial_t [t,a_i])  x^i = 0$. 
Hence we have $(e-i) a_i - (m+1)\partial_t[t,a_i] = 0$ for every $i$. 
Thus we have either $a_i=0$ or $e =\frac{(m+1)\partial_t[t,a_i]}{a_i}+i$.
Since $\frac{(m+1)\partial_t[t,a_i]}{a_i} \in (m+1)\Z + \filt[-1]\hfield$, this implies $a = a_{i_0} x^{i_0}$ for some
  $0\leq i_0\leq m-1$.
\end{proof}

\providecommand{\bysame}{\leavevmode\hbox to3em{\hrulefill}\thinspace}

\end{document}